\documentclass[11pt, reqno, twoside]{amsart}
\usepackage{latexsym}
\usepackage{amssymb}
\usepackage{amsmath}
\usepackage{bbm}
\usepackage{bm}

\usepackage{tikz}
\usepackage{tkz-euclide}
\usetikzlibrary{decorations.fractals}
\usetikzlibrary{decorations.footprints}
\usepackage[colorlinks=true, pdfstartview=FitV, linkcolor=blue, citecolor=magenta, urlcolor=black]{hyperref} % Includes hyperrefs into pdf

\usepackage{amsthm,amstext,amscd,epsfig,euscript,pspicture,multicol,graphpap,graphics,graphicx,enumerate,subfig,sidecap,wrapfig,color,pict2e}
\usepackage{mathrsfs} 
\usepackage{upgreek}
\usepackage{textgreek}

\usepackage[utf8]{inputenc}

\usepackage[top=1in, bottom=1in, left=1in, right=1in, a4paper, marginparwidth=60pt]{geometry}

\usepackage{marginnote}
\calclayout
\allowdisplaybreaks

\theoremstyle{plain}
\newtheorem{theorem}[equation]{Theorem}
\newtheorem{lemma}[equation]{Lemma}
\newtheorem{corollary}[equation]{Corollary}
\newtheorem{proposition}[equation]{Proposition}

\theoremstyle{definition}
\newtheorem{definition}[equation]{Definition}

\theoremstyle{remark}
\newtheorem{remark}[equation]{Remark}

\numberwithin{equation}{section}

\DeclareMathOperator{\dv}{div}
\DeclareMathOperator{\Tr}{Tr}

\DeclareMathOperator{\dist}{dist}

\DeclareMathOperator{\Cov}{Cov}

\DeclareMathAlphabet{\mathsfit}{T1}{\sfdefault}{\mddefault}{\sldefault}
\SetMathAlphabet{\mathsfit}{bold}{T1}{\sfdefault}{\bfdefault}{\sldefault}

\DeclareRobustCommand{\SkipTocEntry}[5]{}%for removing an entry from toc (if hyperref is not involved, change [5] to [4].)

\newcommand{\set}[2]{\left\{#1 : #2\right\}}%---------------------set comprehension
\newcommand{\sub}{\subseteq}%----------------------------subset or equal
\newcommand{\mns}{\setminus}%------------------------------set excision
\newcommand{\N}{\mathbb{N}}%----------------------------------------natural numbers
\newcommand{\R}{\mathbb{R}}%----------------------------------------real numbers
\newcommand{\del}{\partial}%------------------------------------------------topological boundary

\newcommand{\I}{\mathbb{I}}%--------------------------------------------identity matrix
\newcommand{\eps}{\varepsilon}%-----------------------------------------epsilon
\newcommand{\inv}[1]{{#1}^{-1}}%----------------------------------------inverse
\newcommand{\dx}{\, dx}%-----------------------------------------dx after integral
\newcommand{\loc}{\text{\rm loc}}%-----------------------------------loc as subscript
\newcommand{\Om}{\Omega}%--------------------------------------------------------usually domain
\newcommand{\om}{\omega}%---------------------------------------------------usually function
\newcommand{\inp}[2]{\big\langle #1,#2\big\rangle}%-----------------inner product/duality action
\newcommand{\vertiii}[1]{{\left\vert\kern-0.25ex\left\vert\kern-0.25ex\left\vert #1 
    \right\vert\kern-0.25ex\right\vert\kern-0.25ex\right\vert}}%--------------------------------another norm
\newcommand{\gr}{\nabla}%-----------------------------------------------------gradient
\newcommand{\lap}{\Delta}%------------------------------------------------Laplacian
\newcommand{\h}{\mathcal{H}}%------------------------------------------hausdorff measure
\newcommand{\on}{%
  \,\raisebox{-.127ex}{\reflectbox{\rotatebox[origin=br]{-90}{$\lnot$}}}\,}%------------restriction of measure
  
\def\vint_#1{\mathchoice%
          {\mathop{\kern 0.2em\vrule width 0.6em height 0.69678ex
depth -0.58065ex
                  \kern -0.8em \intop}\nolimits_{\kern -0.4em#1}}%
          {\mathop{\kern 0.1em\vrule width 0.5em height 0.69678ex
depth -0.60387ex
                  \kern -0.6em \intop}\nolimits_{#1}}%
          {\mathop{\kern 0.1em\vrule width 0.5em height 0.69678ex
              depth -0.60387ex
                  \kern -0.6em \intop}\nolimits_{#1}}%
          {\mathop{\kern 0.1em\vrule width 0.5em height 0.69678ex
depth -0.60387ex
                  \kern -0.6em \intop}\nolimits_{#1}}}
\def\vintslides_#1{\mathchoice%
          {\mathop{\kern 0.1em\vrule width 0.5em height 0.697ex depth -0.581ex
                  \kern -0.6em \intop}\nolimits_{\kern -0.4em#1}}%
          {\mathop{\kern 0.1em\vrule width 0.3em height 0.697ex depth -0.604ex
                  \kern -0.4em \intop}\nolimits_{#1}}%
          {\mathop{\kern 0.1em\vrule width 0.3em height 0.697ex depth -0.604ex
                  \kern -0.4em \intop}\nolimits_{#1}}%
          {\mathop{\kern 0.1em\vrule width 0.3em height 0.697ex depth -0.604ex
                  \kern -0.4em \intop}\nolimits_{#1}}}

\newcommand{\aveint}[2]{\mathchoice%--------------------------------------------------------integral average on reals
          {\mathop{\kern 0.2em\vrule width 0.6em height 0.69678ex
depth -0.58065ex
                  \kern -0.8em \intop}\nolimits_{\kern -0.45em#1}^{#2}}%
          {\mathop{\kern 0.1em\vrule width 0.5em height 0.69678ex
depth -0.60387ex
                  \kern -0.6em \intop}\nolimits_{#1}^{#2}}%
          {\mathop{\kern 0.1em\vrule width 0.5em height 0.69678ex
depth -0.60387ex
                  \kern -0.6em \intop}\nolimits_{#1}^{#2}}%
          {\mathop{\kern 0.1em\vrule width 0.5em height 0.69678ex
depth -0.60387ex
                  \kern -0.6em \intop}\nolimits_{#1}^{#2}}}

\newcommand{\Sn}{\mathbb S}%------------------------------------------------------------------------------unit sphere
\newcommand{\Snn}{\Sn^{n-1}}%------------------------------------------------------------------------------n-1 dim unit sphere
\newcommand{\g}{\mathbf{g}}%--------------------------------------------------------------------------Gauss map
\newcommand{\F}{\mathscr{F}}%--------------------------------------------------------------------------density of existence measure
\newcommand{\LL}{\mathscr{L}}%--------------------------------------------------------------------------variation of density
\newcommand{\cof}{\mathcal{C}}%-----------------------------------------------------------------------cofactor matrix
\newcommand{\kr}{\mathcal{K}}%-----------------------------------------------------------------------Gauss curvature
\newcommand{\cov}{\text{\raisebox{2pt}{$\bigtriangledown$}}}%------------------------covariant derivative
\newcommand{\covtwo}{{\cov}{}^2}%--------------------------------------second covariant derivative

\newcommand{\Wg}{\mathcal{W}}%---------------------------------------------------------------------Weingarten map
\newcommand{\pnu}{\gr u/|\gr u|}%-------------------------------------gradient normal
\newcommand{\T}{\mathscr{T}}%---------------------------------BM functional

\newcommand{\convexcomb}[2]{(1-\lambda)\, #1 + \lambda\, #2}%----convex combination

\newcommand{\bigbrackets}[1]{\left\{#1\right\}}%----brackets
\begin{document} 

	\allowdisplaybreaks[2]
%\frontmatter

\title[Brunn-Minkowski inequality for $p$-harmonic measures]{Brunn-Minkowski inequality for $p$-harmonic measures}
\author[A. Aguas-Barreno]{Ariel A. Aguas-Barreno}
\address{{Ariel A. Aguas-Barreno}\\
Department of Mathematical Sciences
\\
University of Essex
\\
Colchester CO4 3SQ, United Kingdom}	\email{aa23688@essex.ac.uk}

\author[M. Akman]{Murat Akman}
\address{{Murat Akman}\\
Department of Mathematical Sciences
\\
University of Essex
\\
Colchester CO4 3SQ, United Kingdom}	\email{murat.akman@essex.ac.uk}

\author[S. Mukherjee]{Shirsho Mukherjee}
\address{{Shirsho Mukherjee}\\
Department of Mathematical Sciences
\\
University of Essex
\\
Colchester CO4 3SQ, United Kingdom}	\email{shirsho.mukherjee@essex.ac.uk, \ m.shirsho@gmail.com}

\thanks{M. Akman and S. Mukherjee have been supported by the EPSRC New Investigator award [grant number EP/W001586/1]. A. Aguas-Barreno has been partially supported by the School of Mathematics, Statistics and Actuarial Science at the University of Essex.}

\date{\today}

\subjclass[2020]{35J92, 52A20, 52A40}

\keywords{Convex Geometry, $p$-Laplacian, $p$-harmonic measure, Minkowski problem}

\begin{abstract}
For convex rings with varying inner bodies, we use the reciprocal boundary-gradient inequality, valid for all $1<p<\infty$, to derive pointwise bounds on the densities of the corresponding $p$-harmonic measures for $p\ge n$.
For $p=2$ and $n\ge 3$, we compute the second variation of the associated functional at concentric balls $B_r\Subset B_1$ and determine when it is strictly negative in every nonconstant spherical harmonic direction.
For a fixed inner body, equality in the level-set comparison at some $0<\lambda<1$ for all sufficiently small positive levels implies $K_1=K_2$.
We also give sufficient conditions for a Brunn--Minkowski inequality for the functional on an admissible class of convex domains when $1<p<n+1$.
Under the strictness condition and a one-sided first-variation inequality in both orders, we prove uniqueness of a certain Minkowski problem up to dilation about the origin for the selected measures on the admissible class.
\end{abstract}

\maketitle
\setcounter{tocdepth}{2}

\section{Introduction}

The Brunn--Minkowski inequalities and Minkowski problems for prescribed measures on the unit sphere go back to the works of Minkowski \cite{M2,M1}, Alexandrov \cite{A1, A2}, etc. and have been among the most widely studied topics in geometric analysis. They are also intimately connected to other inequalities including the isoperimetric inequality, Sobolev
inequalities, etc. 
The classical Brunn-Minkowski inequality states that for any compact convex sets $E_1, E_2\subset \R^n$ with nonempty interiors, we have that, 
\begin{align}\label{BMVolume}
|\lambda E_1+(1-\lambda) E_2|^{\frac{1}{n}} 
\geq \lambda |E_1|^{\frac{1}{n}} +(1-\lambda) |E_2|^{\frac{1}{n}},
\end{align}
for any $\lambda \in [0,1]$, where $|\cdot|$ is the Lebesgue measure and $+$ on the left denotes the Minkowski addition. Thus, the inequality \eqref{BMVolume} implies that $|\cdot|^{1/n}$ is a concave function with respect to Minkowski addition. Moreover, for $0<\lambda<1$, equality in \eqref{BMVolume} holds if and only if $E_1$ is a translation and dilation of $E_2$. 
For numerous applications of \eqref{BMVolume} to problems in geometry and analysis, we refer to the classical book by Schneider \cite{Sc} and the survey paper by Gardner \cite{G}.

Over the years, inequalities of Brunn-Minkowski type have also been proved for many other measures and homogeneous functionals. For example, one can replace the Lebesgue measure in \eqref{BMVolume} by capacity. In this case, it was shown by Borell \cite{B1} that
\begin{align}
\label{BMcapacity2}
\left[\mbox{Cap}_{2}(\lambda E_1+(1-\lambda) E_2)\right]^{\frac{1}{n-2}} \geq \lambda \left[\mbox{Cap}_{2}(E_{1})\right]^{\frac{1}{n-2}}+(1-\lambda)\left[\mbox{Cap}_{2}(E_2)\right]^{\frac{1}{n-2}},
\end{align}
whenever $E_1, E_2$ are compact convex sets with nonempty interiors in $\mathbb{R}^{n}$, $n \geq 3$. Here, $\mbox{Cap}_{2}$ denotes the \textit{Newtonian capacity}. The exponents in this inequality and \eqref{BMVolume} differ because $|\cdot |$ is homogeneous of degree $n$, whereas $\mbox{Cap}_{2}(\cdot)$ is homogeneous of degree $n-2$. 
Borell \cite{B2} proved a Brunn-Minkowski-type inequality for \textit{logarithmic capacity}. The equality case in \eqref{BMcapacity2} was studied by Caffarelli, Jerison, and Lieb in \cite{CJL}, where it was shown that equality in \eqref{BMcapacity2}, for $0<\lambda<1$, holds if and only if $E_2$ is a translate and dilate of $E_1$ for $n \geq 3$. Jerison in \cite{J2} used that result to prove uniqueness of the Minkowski problem. In \cite{CS}, Colesanti and Salani proved the $p$-capacitary version of \eqref{BMcapacity2} for $1 < p < n$, i.e.  
\begin{align}
\label{BMcapacityp}
\left[\mbox{Cap}_{p}(\lambda E_1+(1-\lambda) E_2)\right]^{\frac{1}{n-p}} \geq \lambda \left[\mbox{Cap}_{p}(E_{1})\right]^{\frac{1}{n-p}}+(1-\lambda)\left[\mbox{Cap}_{p}(E_2)\right]^{\frac{1}{n-p}}
\end{align}
whenever $E_1, E_2$ are compact convex sets with nonempty interiors in $\mathbb{R}^{n}$, where $\mbox{Cap}_{p}(\cdot)$ denotes the $p$-capacity of a set defined as
\begin{align}
    \label{BM_pCap}
    \mbox{Cap}_{p}(E) = \inf \left\{ \int_{\mathbb{R}^{n}} |\nabla v|^{p} \, dx : v \in C^{\infty}_{0}(\mathbb{R}^{n}), \, v(x) \geq 1 \text{ for } x \in E \right\}.
\end{align}
It was also shown in the same paper that equality in \eqref{BMcapacityp}, for $0<\lambda<1$, holds if and only if $E_2$ is a translate and dilate of $E_1$. In \cite{CC}, Colesanti and Cuoghi defined a logarithmic capacity for $p = n, n \geq 3,$ and proved a Brunn-Minkowski-type inequality for this capacity. In \cite{CNSXYZ}, a Minkowski problem was studied for $p$-capacity, $1 < p < 2$, using \eqref{BMcapacityp}. For $1 < p < n$, the second author of the present article, along with others in \cite{AGHLV}, showed that \eqref{BMcapacityp} holds for any convex set with positive $p$-capacity and for capacities associated with more general elliptic PDEs (known as $\mathcal{A}$-harmonic PDEs), with the same conclusion if equality holds. In \cite{ALSV}, the existence and uniqueness of the so-called $\mathcal{A}$-harmonic Green's function for the complement of a convex compact set in $\mathbb{R}^n$ were established. Using this, it was shown that a quantity related with the behavior of this function near infinity satisfies a Brunn-Minkowski inequality for $n \leq p < \infty$. It was also shown in the same article that if equality holds for $E_1$ and $E_2$, then under certain regularity and structural assumptions on $\mathcal{A}$ (no additional assumption in the case of $p$-Laplace equation), these two sets are homothetic. 
The torsional rigidity and first eigenvalue of the Laplacian versions of \eqref{BMVolume} were studied in \cite{C}. We refer to \cite{ALSV,AGHLV,ALVcalvarpde} etc. for more results. 

Harmonic measures have been studied by Dahlberg \cite{D} and Jerison \cite{J1} and as a generalization of it, the $p$-harmonic measures have been studied by Lewis and Nystr\"{o}m, see \cite{LN,LNld, LNV}, etc. (see Section \ref{sec:prelim} for details). 
The generalization to $p$-harmonic measures incorporates a larger class of measures within which the classical harmonic measures as in \cite{J1} are special cases even within $p=2$.

Our first result gives sufficient conditions for a Brunn--Minkowski inequality for a functional associated with $p$-harmonic measures. To state it, we let $K_0,L_0$ be smooth strongly convex bodies with $0\in\mathring L_0\Subset\mathring K_0$, and let $u_0$ solve the convex-ring Dirichlet problem with values $0$ on $\del K_0$ and $\eps_0>0$ on $\del L_0$. For the Dirichlet problem defined in \eqref{eq:hatdir}, the inner body of $K$ is
\[
L(K)=s(K)L_0,\qquad
s(K)=\frac{\displaystyle\int_{\Snn}h_K\,d\xi}{\displaystyle\int_{\Snn}h_{K_0}\,d\xi}.
\]
When the inner body needs to be displayed, we write $u_{K,L}$, $\mu_{K,L}$, $R_{K,L}$, $\om_{K,L}$ and $\T(K;L)$ for the solution, spherical measure, spherical density, boundary measure and functional, respectively, with inner value $\eps_0$. In the main theorems, the inner body is $L=L(K)$, and we denote these quantities as $u_K$, $\mu_K$, $R_K$, $\om_K$ and $\T(K)$. 
We next state our first result. 
  \begin{theorem}\label{thm:bm}

Let $n\geq2$, $1<p<n+1$, $q=n-p+1$, and let the smooth convex-ring datum $(K_0,L_0,u_0)$ be as above. Let $\mathcal D$ be an admissible class in the sense of Definition~\ref{def:normalizedclass}. For $K\in\mathcal D$, define
\begin{equation}\label{eq:bmfunc}
\T(K) = \int_{\del K} (h_K\circ \g_K) \, d\om_K
=\int_{\Snn}h_K\,d\mu_K,
\end{equation}
where $\om_K$ is induced by \eqref{eq:hatdir} and $\mu_K=(\g_K)_*\om_K$. We have
\begin{equation}\label{eq:bm0}
\T\big((1-\lambda)K_{1} + \lambda K_{2}\big)^{1/q}
\geq (1-\lambda)\T(K_1)^{1/q}+\lambda\T(K_2)^{1/q},
\end{equation}
for every $K_1,K_2\in\mathcal D$ and $\lambda\in[0,1]$. Moreover, if the class satisfies the strictness condition in Definition~\ref{def:normalizedclass}, equality holds for some $\lambda\in(0,1)$ if and only if $K_1$ and $K_2$ are dilates about the chosen origin.

\end{theorem}
For the functional $\T(\cdot;L_0)$ with fixed inner body $L_0$, the first step is to express the $p$-Laplacian on convex rings in terms of the support functions of superlevel sets. Proposition~\ref{prop:plaphu} uses the support-function formulas of Colesanti--Salani~\cite{CS} and Longinetti--Salani~\cite{LS}.
The linearity of support functions with respect to Minkowski addition then gives a subsolution on the convex combination in Subsection~\ref{subsec:sup}.
We also obtain an exact limiting characterization of the functional $\T(\cdot;L_0)$. In the capacitary case of \cite{CS}, the identity
\[
\mbox{Cap}_{p}(\Om)=c(n,p)\left(\lim_{|x|\to\infty}\frac{u(x)}{G(x)}\right)^{p-1},
\qquad G(x)=\frac{1}{|x|^{(n-p)/(p-1)}},
\]
follows from the asymptotic behavior of the $p$-capacitary function at infinity. For our case, the vanishing boundary is finite and therefore it is different from the capacitary case. Using boundary regularity of $p$-harmonic functions and the normal projection onto the boundary, we prove in Proposition \ref{prop:wlimchar}, the following,
\[
\T(\Om;L_0)=\lim_{s\to0^+}\int_{\partial\Om_s}
h_{\Om_s}(\g_{\Om_s}(x))
\left(\frac{u(x)}{\dist(x,\partial\Om)}\right)^{p-1}
\,d\mathcal H^{n-1}(x).
\]
Here $u=u_{\Om,L_0}$.
Combining this characterization with the supremal-convolution comparison and the geometric and covariance conditions gives Proposition~\ref{prop:condbm}. Remark~\ref{thm:bm1} follows directly from the boundary covariance identity. For the functional $\T$ with inner body $L(K)$, the boundary identity in Proposition~\ref{prop:boundaryidentity}, the homogeneity in Lemma~\ref{prop:hathom}, and the conditions of Definition~\ref{def:normalizedclass} give Theorem~\ref{thm:bm}. For the problem with fixed inner body $L_0$, Proposition~\ref{prop:conditionalhomothety} shows that equality of the level-set support functions forces the outer bodies to coincide. Under the hypotheses of Proposition~\ref{prop:eqcase}, the required level-set equality follows.

As a consequence of Theorem \ref{thm:bm}, subject to the one-sided first-variation inequality below, we have the following uniqueness theorem.
\begin{theorem}\label{thm:pharminkunique}

Assume the hypotheses and strictness condition of Theorem \ref{thm:bm}. For $K_1,K_2\in\mathcal D$, assume also that the one-sided derivatives satisfy 
\begin{equation}\label{eq:Thadamard}
\left.\frac{d}{d\nu}\right|_{0+}
\T((1-\nu)K_i+\nu K_j)
\leq q\int_{\Snn}(h_{K_j}-h_{K_i})\,d\mu_{K_i},
\qquad (i,j)=(1,2),(2,1).
\end{equation}
If $\mu_{K_1}=\mu_{K_2}$, then $K_1,K_2$ are dilates about the chosen origin. If $p\ne n$, they coincide.

\end{theorem}
It is natural that the boundary regularity of $p$-harmonic functions and the
geometry of the boundary on which they vanish play a key role in such problems. For any smooth strongly convex ring $L\Subset\mathring K$, the results of Lewis~\cite{Lewis77}, Hopf's lemma, Schauder theory~\cite{GT}, and Ma--Ou--Zhang~\cite{Ma-Ou-Zh} ensure automatically that $u_{K,L}\in C^\infty(\overline{\mathring K\setminus L})$ has nonvanishing gradient and strictly convex level surfaces with positive Gaussian curvature. Using supremal convolution for rings with varying inner bodies $L_\lambda=(1-\lambda)L_1+\lambda L_2$, we obtain in Corollary~\ref{cor:gradconv} unconditional convexity of $|\gr u_{K,L}(\gr h_K(\xi))|^{-1}$ for all $1<p<\infty$ and unconditional pointwise lower bounds on $R_{K,L}$ for $p\ge n$. Furthermore, for $p=2$ and $n\ge 3$ near concentric balls $B_r\Subset B_1$, Proposition~\ref{prop:p2threshold} computes the second variation in spherical harmonic directions. We give further remarks on the conditions and the restriction $p<n+1$ at the end.

\section{Preliminaries}\label{sec:prelim}

Here, we review well-known properties of convex domains,  $p$-harmonic measures and Minkow\-ski problems. For $n\geq 1$, points in $\R^n$ are denoted as
$ x = ( x_1,
\dots,  x_n) =\sum_{i=1}^n x_ie_i$ where $\{e_1,\ldots,e_n\}$ is the standard basis. The standard inner
product on $ \mathbb R^{n} $ is denoted as $\inp{\cdot}{\cdot} $ and the Euclidean norm as $|\cdot|$. 
For a scalar function $f:\R^n\to\R$, write $\gr f=\sum_{i=1}^n(\del_i f)e_i$. For a vector-valued function $F=(f_1,\ldots,f_n):\R^n\to\R^n$, the Jacobian and scalar Hessian conventions are
$$ DF = \sum_{i,j=1}^n (\del_{i} f_j) e_i\otimes e_j
\quad\text{and}\quad D^2f = \sum_{i,j=1}^n (\del^2_{i,j} f) e_i\otimes e_j.
$$
Thus $DF$ is the transpose of the usual Jacobian.
For any $A\subset\R^n$, its topological closure and interior shall be denoted as $\bar A$ and 
$\mathring A$ and $\del A=\bar A\mns \mathring A$.

Given a measure space $(X, \mu)$ and 
a measurable map $f:X\to Y$, the pushforward measure $f_*\mu$ on $Y$ is defined on any measurable subset $ E\sub Y$ as
$$ (f_* \mu)(E) = \mu(\inv{f}(E)),$$
and satisfies $\int_Y g\,d(f_*\mu)=\int_X g\circ f\,d\mu$ for every nonnegative measurable or integrable $g$. We shall denote $\h^k$ as the $k$-dimensional Hausdorff measure, $B_r(x)$ as the standard metric ball in $\R^n$ and 
$\Snn =\del B_1(0)$. We shall denote $\upomega_n=\h^{n-1}(\Snn)$ as the surface area of the unit sphere.

\subsection{Convex domains and support functions}\label{subsec:convsupp}
We recall some notions and properties of convex functions and domains. For further details we refer to Schneider \cite{Sc}, Gardner \cite{G}, etc.  For any nonempty bounded $A\subset \R^n$ , the {support function} $h_A: \R^n \to \R$ is defined as 
\begin{equation}\label{eq:suppf}
h_A( y) = \sup \set{\inp{x}{ y} }{ x\in A}.
\end{equation}
Thus, we have $h_K( y) = \max_{x\in K}\inp{x}{ y}$ for a compact convex subset $K\subset\R^n$. Note that $h_K$ is sub-linear in general and linear (i.e. $h_K( y)= \inp{x}{ y}$ for every 
$ y\in \R^n$) iff $K=\{x\}$. The {supporting hyperplane} of $K$ with outer normal $ y\in  \R^n\mns\{0\}$, is given by 
\begin{equation}\label{eq:supphp}
H_K ( y):= \set{x\in \R^n}{\inp{x}{ y}= h_K( y)}.
\end{equation}
Thus, $h_K$ is differentiable at $ y\in \R^n\mns\{0\}$ with $\gr h_K( y)=x$ if and only if $K\cap H_K( y)=\{x\}$. If this holds for every $y\ne0$, then $K$ is strictly convex. Positive homogeneity gives 
\begin{equation}\label{eq:homgradhom}
 h_K (\lambda  y) = \lambda h_K( y)\quad\text{and}\quad
\gr h_K (\lambda  y) = \gr h_K( y),
\end{equation}
where $\lambda>0$, the first identity holds for every $y\ne0$, and the second holds at each differentiability direction.
The support functions are sub-linear (sub-additive and $1$-homogeneous) and also conversely, all finite sub-linear functions are support functions for a corresponding nonempty compact convex set, possibly with empty interior, see \cite[Theorem 1.7.1]{Sc}. 
This leads to characterizing convex set $K$ as
\begin{equation}\label{eq:kinthk}
K = \bigcap_{ \xi\in \Snn} \{x\in \R^n : \inp{x}{ \xi}\leq  h_K( \xi)\}.
\end{equation}
Thus, convex bodies correspond to Wulff shapes of their support functions. 
Also note that \eqref{eq:kinthk} implies 
$K_1\sub K_2$ iff $h_{K_1}\leq h_{K_2}$. 
Recall that the {Minkowski sum} of $E,F\sub \R^n$ is defined by 
\begin{equation}\label{eq:minksum}
 E  +  F  := \{ x + y :  x \in E, y \in F   \}
\end{equation}
and the scalar multiplication $ c E :=  \{ c y : y \in E\} $ for any $c\in \R$. It is not hard to check that if $E, F$ are convex then so is $E+F$ and $\alpha E$ for all $\alpha\geq 0$. 
In particular, $ B_r(x)= \{x\}+ r\,B_1(0)$ is the standard Euclidean metric ball. 
Most importantly, the support functions are Minkowski additive, i.e. for compact convex sets $E, F\subset \R^n$, we have 
\begin{equation}\label{eq:suppminkadd}
h_{\alpha E+\beta F} = \alpha h_E + \beta h_F \qquad\forall\ \alpha, \beta\geq 0, 
\end{equation}
and hence, for supporting hyperplanes as in \eqref{eq:supphp}, $H_{\alpha E+\beta F}(\cdot) = \alpha H_E (\cdot)+ \beta H_F(\cdot)$ holds as well when $\alpha+\beta>0$. 
The classical Brunn-Minkowski inequality is given by
\begin{equation}\label{eq:bm}
|E+F|^{1/n} \geq |E|^{1/n}+ |F|^{1/n}, 
\end{equation}
for convex bodies $E,F\sub\R^n$ of positive volume, with equality if and only if $E$ and $F$ are homothetic.
The notion of smoothness of boundaries is standard, a domain $\Om\subset \R^n$ is said to be of class $C^k$ (resp. $C^{k,\alpha}$ for $k\in \N$ and $0<\alpha\leq 1$) if 
the domain can be expressed locally as a sub-graph of a function 
$\phi\in C^{k}(\R^{n-1})$ (resp. $C^{k,\alpha}(\R^{n-1})$) and the boundary as a graph of it, 
after a possible rotation.
If $\del K$ is of class $C^1$ for a convex domain $K$, then for every $x\in \del K$ there exists a normal vector at $x$ unique up to scaling. This gives rise to the {Gauss map}, 
$\g_K : \del K\to \Snn$, where $\g_K(x)$ is the outer unit normal at $x\in \del K$ and in terms of support function, we have
\begin{equation}\label{eq:supg}
h_K (\g_K(x)) = \inp{x}{\g_K(x)}, \quad \forall\ x\in\del K. 
\end{equation}
Recall that the tangent space at $x\in\del K$ is given by 
$ T_x(\del K)= \set{y\in \R^n}{\inp{y}{\g_K(x)}=0} $
and from \eqref{eq:supg}, 
if in addition $\del K$ is $C^2$, we have $\g_K\in C^1(\del K, \Snn)$, and
the {Weingarten map} is defined by 
$$\Wg_K(x):= d\,\g_K(x) : T_x(\del K)\to T_x(\del K).$$ The {principal curvatures} are the eigenvalues 
$\kappa_1(x),\ldots, \kappa_{n-1}(x)$ of $\Wg_K(x)$ 
and the {Gaussian curvature} 
is given by
$$ %H(x)=\frac{\Tr(\Wg_K(x))}{n-1}
%= \frac{1}{n-1}\big(\kappa_1+\ldots+ \kappa_{n-1}\big)\quad \text{and}\quad
 \kr_{\del K}(x)=\det(\Wg_K(x))= \kappa_1(x)\ldots\kappa_{n-1}(x).$$  
If $K$ is strictly convex then $K\cap H_K(\g_K(x))= \{x\}$ for any $x\in \del K$ and $\inv{\g_K}:\Snn\to \del K$ is well defined. Moreover, 
$h_K$ being differentiable everywhere in $\R^n\mns\{0\}$, from \eqref{eq:supg} we have 
\begin{equation}\label{eq:hgrh}
h_K( \xi) = \inp{ \xi}{\inv{\g_K}( \xi)} \quad\text{and}\quad \gr h_K( \xi) = \inv{\g_K}( \xi), \quad\forall\  \xi\in \Snn. 
\end{equation}
A convex set $K$ is said to be of class $C^k_+$ 
for any integer $k\geq2$,  
if $\del K$ is of class $C^k$ 
and the Gauss map 
$\g_K : \del K\to \Snn$ is a diffeomorphism. 
Given any nonempty compact convex set $K$, there exists a nested sequence $K_i$ of $C^2_+$ convex sets with $K_{i+1}\subset K_i$ such that  
\begin{equation}\label{eq:c2ap}
K=\bigcap_{i=1}^\infty K_i,
\end{equation}
(see \cite{Sc}).
If $K$ is of class $C^2_+$, note that $\inv{\g_K}\in C^1(\Snn, \del K)$ and $h_K\in C^2(\R^n\mns\{0\})$. Hence, 
from homogeneity of $h_K$, \eqref{eq:hgrh} and \eqref{eq:homgradhom}, we have 
\begin{equation}\label{eq:hgradh}
h_K( y) = \inp{ y}{\gr h_K( y)}\quad \text{and}\quad D^2h_K( y)  y =0, \qquad\forall\  y\in \R^n\mns\{0\}, 
\end{equation}
where the second equality of the above follows from the first. 
The inverse Weingarten map is denoted as 
$$\inv{\Wg_K} ( \xi):= \inv{(\Wg_K(\inv{\g_K}( \xi)))}= d\,\inv{\g_K}( \xi) : T_ \xi (\Snn) \to T_ \xi (\Snn)$$ defined for all $ \xi\in \Snn$, which is the non-singular part of $D^2h_K( \xi)$. 
For any $ \xi\in \Snn$ being fixed, there exists an orthonormal basis $\{e^1,\ldots,e^{n-1},  \xi\}$ of $\R^n$ and hence, $\{e^i\}$ span the tangent space 
$T_\xi (\Snn)$. Any $x\in \R^n$ in this basis is given by 
\begin{equation}\label{eq:basis}
x= \sum_{k=1}^{n-1}x^k e^k + \inp{x}{ \xi} \xi\qquad\text{with}\quad x^k= \inp{x}{e^k}.
\end{equation}
Since $D^2h_K ( \xi)  \xi =0$ from \eqref{eq:hgradh}, the only
non-zero entries of $D^2h_K ( \xi)$ are $\inp{D^2 h_K( \xi) e^j}{e^i}$ for $i,j\in\{1, \ldots, n-1\}$ which, from \eqref{eq:hgrh}, are 
the entries of $d \gr h_K( \xi)= d\,\inv{\g_K}( \xi)=\inv{\mathcal W_K}( \xi)$, with respect to the above basis. In other words, we have %$D^2 h_K( \xi) = \sum_{i,j=1}^{n-1} \inp{\inv{\mathcal W_K}( \xi) e^j}{e^i} e^i \otimes e^j$
\begin{equation}\label{eq:d2wing}
D^2 h_K( \xi) = \sum_{i,j=1}^{n-1} \inp{\inv{\mathcal W_K}( \xi) e^j}{e^i} e^i \otimes e^j. 
\end{equation}
Thus, $D^2 h_K( \xi)$ has eigenvalues $\{1/\kappa_1(\inv{\g_K}( \xi)),\ldots, 1/\kappa_{n-1}(\inv{\g_K}( \xi)), 0\}$, where $\kappa_i$'s are the principal curvatures of $\del K$.
We write $\cov f$ and $\covtwo f$ for the intrinsic gradient and Hessian of a function $f$ on $\Snn$. Given $\xi\in\Snn$ and the orthonormal basis $\{e^1,\ldots,e^{n-1},\xi\}$ of \eqref{eq:basis}, we use the chart of the open hemisphere centred at $\xi$,

\[
\psi(z)=\sum_{i=1}^{n-1}z_ie^i+\sqrt{1-|z|^2}\;\xi,\qquad z\in\R^{n-1},\ |z|<1 .
\]
Then $\psi(0)=\xi$, $\del_i\psi(0)=e^i$ and $\del_{i,j}\psi(0)=-\delta_{i,j}\,\xi$. The metric coefficients $\inp{\del_i\psi(z)}{\del_j\psi(z)}=\delta_{i,j}+z_iz_j/(1-|z|^2)$ equal $\delta_{i,j}$ at $z=0$, and their first derivatives vanish there. Hence the Christoffel symbols of $\psi$ vanish at $z=0$, and at the centre $\xi$ we have
\begin{equation}\label{eq:cov}
\begin{aligned}
(i)&\ \cov f(\xi)=\sum_{i=1}^{n-1}\cov_i f(\xi) \, e^i, \quad \text{where}\quad \cov_if (\xi):= \del_i (f\circ \psi)(0),\\
(ii)&\ \covtwo f(\xi)=\sum_{i,j=1}^{n-1} \cov_{i,j} f(\xi)\, e^i \otimes e^j,  \quad \text{where}\quad \cov_{i,j}f (\xi):= \del_{i,j} (f\circ \psi)(0).
\end{aligned}
\end{equation}
Covariant derivatives at a point are computed in the chart centred at that point. For a function defined on $\R^n\mns\{0\}$, $\cov$ and $\covtwo$ refer to its restriction to $\Snn$. The following identities also appear in \cite{CNSXYZ}. We include their proof for completeness.
\begin{lemma}\label{lem:sphder}
Let $f\in C^1(\R^n\mns\{0\})$ satisfy $f(sy)=sf(y)$ for all $s>0$ and $y\neq0$, and let $\xi\in\Snn$.
\begin{enumerate}
\item $\gr f(\xi)=f(\xi)\,\xi+\cov f(\xi)$.
\item If $f\in C^2(\R^n\mns\{0\})$, then $D^2f(\xi)\xi=0$ and $\inp{D^2f(\xi)e^i}{e^j}=\cov_{i,j}f(\xi)+f(\xi)\,\delta_{i,j}$ for $i,j\in\{1,\ldots,n-1\}$.
\end{enumerate}
In particular, if $K$ is of class $C^2_+$, then for every $\xi\in\Snn$,
\begin{equation}\label{eq:grhcov}
\gr h_K ( \xi)= h_K( \xi) \xi + \cov h_K( \xi), \quad\text{and}\quad 
 \inv{\mathcal W_K}( \xi) = h_K( \xi)\I +\covtwo  h_K ( \xi).
\end{equation}
\end{lemma}
\begin{proof}
Differentiating $f(s\xi)=sf(\xi)$ at $s=1$ gives $\inp{\gr f(\xi)}{\xi}=f(\xi)$. By the chain rule, $\cov_if(\xi)=\del_i(f\circ\psi)(0)=\inp{\gr f(\xi)}{e^i}$, and (1) follows from \eqref{eq:basis}. If $f\in C^2(\R^n\mns\{0\})$, then $\gr f(s\xi)=\gr f(\xi)$ for $s>0$, and differentiating at $s=1$ gives $D^2f(\xi)\xi=0$. Differentiating $f\circ\psi$ twice at $z=0$ and using $\inp{\gr f(\xi)}{\xi}=f(\xi)$, we obtain
\[
\cov_{i,j}f(\xi)=\inp{D^2f(\xi)e^i}{e^j}+\inp{\gr f(\xi)}{\del_{i,j}\psi(0)}
=\inp{D^2f(\xi)e^i}{e^j}-\delta_{i,j}f(\xi).
\]
If $K$ is of class $C^2_+$, then $h_K\in C^2(\R^n\mns\{0\})$, and \eqref{eq:grhcov} follows from (1), (2) and \eqref{eq:d2wing}.
\end{proof}

Therefore, for $C^2_+$ domains, we have 
\begin{equation}\label{eq:detkinv}
 \det\big(\covtwo  h_K( \xi) + h_K( \xi) \I\big) = \det (\inv{\mathcal W_K}( \xi))= 1/\kr(\inv{\g_K}( \xi)), 
\end{equation}
where $\kr$ is the Gaussian curvature. Hence, 
%One can show that a domain is $C^2_+$ iff it is $C^2$ and $\kr >0$. 
from transformation rule of the Jacobian, 
% $$(\g_K)_*\h^{n-1}\on_{\del K}= |\det(d\,\inv{\g_K})| \h^{n-1}\on_{\Snn}=1/(\kr\circ \inv{\g_K}) \h^{n-1}\on_{\Snn}$$ and from \eqref{eq:detkinv}, for any $f \in L^1(\del K, \h^{n-1})$, 
 we have 
\begin{equation}\label{eq:intbd}
\int_{\del K} f(x) \, d\h^{n-1}(x)= \int_{\Snn} f(\inv{\g_K}( \xi)) \det\big(\covtwo  h_K( \xi) + h_K( \xi) \I\big) \, d \xi. 
\end{equation}
For a symmetric matrix $A$, we denote by $\cof[A]$ its cofactor matrix. It is symmetric and satisfies $\cof[A]A=\det(A)\I$. Hence, from \eqref{eq:detkinv},
\begin{equation}\label{eq:cofdet}
\cof[\covtwo  h_K + h_K \I] (\covtwo  h_K + h_K \I) = \det(\covtwo  h_K + h_K \I)\I = 1/(\kr \circ \inv{\g_K})\I.
\end{equation}
Since $\inv{\mathcal W_K}= \covtwo  h_K + h_K \I$ from \eqref{eq:grhcov}, hence \eqref{eq:cofdet} 
can be rewritten as 
\begin{equation}\label{eq:cofwin}
\mathcal W_K(x) = \kr (x) \,\cof[\covtwo  h_K + h_K \I](\g_K(x)), \qquad \ \forall\ x\in \del K. 
\end{equation}

The next lemma expresses the Hessian of a function on a regular level set in terms of the support function of that set.
\begin{lemma}\label{lem:grad2u} 
Let $\Om$ be a bounded convex domain of class $C^2_+$ and $h=h_\Om$. Let $u$ be of class $C^2$ on an open set containing $\del\Om$, constant on $\del\Om$, with $|\gr u|>0$ and $\gr u=-|\gr u|\,\g_\Om$ on $\del\Om$. For $\xi\in\Snn$, let $F(\xi)=\inv{\g_\Om}(\xi)=\gr h(\xi)$, and let $\{e^1,\ldots,e^{n-1}\}$ be an orthonormal basis of $T_\xi(\Snn)=T_{F(\xi)}(\del\Om)$. Then, with $\cof_{i,j}[\cdot]=\inp{\cof[\cdot]e^j}{e^i}$ and the covariant derivatives \eqref{eq:cov} taken at $\xi$, we have the following:
\begin{enumerate}
\item $\inp{D^2 u(F( \xi)) e^i}{e^j}=-|\gr u(F(\xi))|\inp{\Wg_\Om(F(\xi))e^i}{e^j}$, and hence
\[
\inp{D^2 u(F( \xi)) e^i}{e^j}= -\kr(F( \xi))\, |\gr u(F( \xi))| \,\cof_{i,j}[\covtwo  h + h \I](\xi);
\]
\item $\inp{D^2 u(F( \xi)) \xi}{e^i} = -\kr(F( \xi)) \sum_{j}\cof_{i,j}[\covtwo  h + h \I](\xi)\,\cov_j (|\gr u|\circ F)(\xi)$.
\end{enumerate}
\end{lemma}
\begin{proof}
Fix $\xi$, and put $x=F(\xi)$ and $\Wg=\Wg_\Om(x)$. Since $\del\Om$ is of class $C^2$ and $|\gr u|>0$ near $\del\Om$, the maps $\g_\Om$ and $|\gr u|$ are of class $C^1$ near $x$. Let $X\in T_x(\del\Om)$, and let $\gamma$ be a $C^1$ curve in $\del\Om$ with $\gamma(0)=x$ and $\gamma'(0)=X$. Differentiating $\gr u(\gamma(s))=-|\gr u(\gamma(s))|\,\g_\Om(\gamma(s))$ at $s=0$, and using $d\g_\Om(x)X=\Wg X$, we get
\begin{equation}\label{eq:d2uX}
D^2u(x)X=-\inp{\gr(|\gr u|)(x)}{X}\,\xi-|\gr u(x)|\,\Wg X .
\end{equation}
Taking $X=e^i$ and the inner product with $e^j$, which is orthogonal to $\xi$, gives the first identity in (1), and the second follows from \eqref{eq:cofwin}. Next, $F$ is of class $C^1$ with $dF(\xi)=\Wg^{-1}$, so the chain rule gives
$\cov_j(|\gr u|\circ F)(\xi)=\del_j(|\gr u|\circ F\circ\psi)(0)=\inp{\gr(|\gr u|)(x)}{\Wg^{-1}e^j}$.
Taking $X=\Wg^{-1}e^j$ in \eqref{eq:d2uX} and the inner product with $\xi$, we obtain $\inp{D^2u(x)\Wg^{-1}e^j}{\xi}=-\cov_j(|\gr u|\circ F)(\xi)$. Since $e^i=\sum_{j}\inp{\Wg e^i}{e^j}\,\Wg^{-1}e^j$ and $D^2u(x)$ is symmetric,
\[
\inp{D^2u(x)\xi}{e^i}=-\sum_{j=1}^{n-1}\inp{\Wg e^i}{e^j}\,\cov_j(|\gr u|\circ F)(\xi),
\]
and \eqref{eq:cofwin} gives (2).
\end{proof}

The {Hausdorff distance} between Borel sets $ E, E' \subset\mathbb R^n $ 
is defined as 
 \begin{equation}\label{eq:defdh}
 d_{\h} ( E, E' ) = \max  \Big(  \sup   \{ \dist ( y, E ) : y \in E' \}\, ,\,   \sup \{ \dist ( y, E' ) : y \in E \} \Big); 
\end{equation}  
 equivalently, we have $ d_{\h} ( E, E' ) = \sup_{y\in \R^n} |\dist ( y, E )-\dist ( y, E' )|$. 
 The Hausdorff distance of convex sets can be characterized by support functions as 
\begin{equation}\label{eq:hdh}
d_\h (E, F)=d_\h (\del E, \del F) = \|h_E -h_F\|_{L^\infty(\Snn)},
\end{equation}
see \cite[Lemma 1.8.1, Lemma 1.8.14]{Sc}. 

\begin{remark}\label{thm:c2approx}
For convex bodies, Hausdorff convergence $K_j\to K$ implies weak convergence of their surface-area measures $S_{K_j}\rightharpoonup S_K$. Smooth strictly convex approximations may be chosen, see \cite{Sc}. If, in addition, $h_{K_j}\to h_K$ in $C^2(\Snn)$ and $\covtwo h_K+h_K\I$ is positive definite, then
\[
\frac{1}{\kr_{\del K_j}(\gr h_{K_j}(\xi))}
=\det(\covtwo h_{K_j}+h_{K_j}\I)
\longrightarrow\det(\covtwo h_K+h_K\I)
=\frac{1}{\kr_{\del K}(\gr h_K(\xi))}
\]
uniformly on $\Snn$. The curvatures themselves also converge uniformly and remain bounded above and below by positive constants for large $j$.

The curvature-radius matrices converge uniformly. Compactness of $\Snn$ bounds their eigenvalues uniformly away from zero, so continuity of the determinant and inversion gives the claims.
\end{remark}

\subsection{\texorpdfstring{The $p$-Laplacian and $p$-harmonic measures}{The p-Laplacian and p-harmonic measures}}\label{subsec:pharm}  
Here we recall $p$-harmonic measures and Minkowski problems and review some previous results in these directions.  The minimizers of the $p$-Dirichlet energy functional $w\mapsto \int_\Om |\gr w|^p \dx $ are weak solutions to the $p$-Laplacian equation $\lap_p u= \dv(|\gr u|^{p-2}\gr u)=0$ in $\Om$, satisfying
$$ \int_\Om |\gr u|^{p-2}\inp{\gr u}{\gr \phi}\dx = 0 ,\qquad \forall\ \phi\in C^\infty_0(\Om),$$
 and are called $p$-{harmonic functions} for $1<p<\infty$ which coincide with harmonic functions for $p=2$. 
We collect these in the following and refer to \cite{HKM, GT, Dib, L2, T, Li}, etc. for more details. 
\begin{theorem}\label{thm:plap}
For any bounded domain $\Om\subset \R^n$ and $1<p<\infty$, we have the following:
\begin{enumerate}
\item (Comparison Principle) if $\lap_p v\leq \lap_p u$ in the weak sense on $\Om$ and $u\leq v$ on $\del\Om$ in the trace sense then $u\leq v$ on $\Om$;
\item (Regularity) if $\lap_p u=0$ in the weak sense on $\Om$, then 
$u\in C^{1,\alpha}_\loc(\Om)$ for some $\alpha\in (0,1)$. On a $C^{1,\gamma}$ domain with $C^{1,\gamma}$ Dirichlet data, the solution has $C^{1,\beta}(\bar\Om)$ regularity for an appropriate $\beta>0$, see \cite{Li}.
\end{enumerate}
\end{theorem}
Where $\gr u\ne0$, the equation is locally uniformly elliptic. Applying difference quotients to the weak equation shows that each first derivative of $u$ solves the linearized divergence-form equation with locally H\"older continuous coefficients. Interior Schauder estimates and bootstrapping then give $u\in C^\infty$ there, see \cite{GT}.

Given a bounded and sufficiently regular domain $\Om\subset \R^n$,  and a neighbourhood $N$ of $\del\Om$, let $u\in W^{1,p}(\Om\cap N)$ be a positive weak solution of $\lap_pu=0$ in $\Om\cap N$, with $u=0$ on $\del\Om$ in the trace sense. Extend $u$ by zero across $\del\Om$. For the smooth domains considered here, this extension belongs to $W^{1,p}(N)$, and there is a nonnegative Radon measure $\om_p$ supported on $\del\Om$ such that

$$ \int_{\Om\cap N} |\gr u|^{p-2}\inp{\gr u}{\gr \phi}\dx = -\int_{\del\Om} \phi \, d\om_p ,$$

for any $\phi\in C^\infty_0(N)$. The measure $\om_p$ is called the $p$-{harmonic measure} associated to $u$.
Such measures can also be defined for $\mathcal A$-harmonic functions and are generally referred to as Riesz measures (see \cite{HKM, KZ}).
If the outer boundary is smooth, $u$ is $C^1$ up to it, and $|\gr u|>0$ there, then $\nu_{\del\Om}= -\gr u/|\gr u|$ and integration by parts gives
\begin{equation}\label{eq:phm}
d\om_p= |\gr u|^{p-1} d\h^{n-1}\on_{\del\Om}.
\end{equation}
The function $u$ may have a blow-up inside 
$\Om\mns N$ but the $p$-harmonicity in $\Om \cap N$ implies that $u\in C^{1,\alpha}_\loc(\Om \cap N)$ from Theorem \ref{thm:plap} above. For the differential arguments below we assume smooth boundary data and nonvanishing gradient on the closure of an interior neighborhood of the outer boundary. It follows that $u$ is smooth in this interior neighborhood. We also require
$$\|u\|_{L^\infty(\del N\cap \Om)}+\|\gr u\|_{L^\infty(\del N\cap \Om)} <\infty ,$$
in other words, all possible singularities of $u$ are strictly in the interior of $\Om\mns \bar N$.
For a convex body $K$, take $\Om=\mathring K$ and let $u\in W^{1,p}(\Om\cap N)$ be a positive weak $p$-harmonic function with zero trace on $\del K$. Under the boundary regularity assumptions of \eqref{eq:phm}, the associated measure has density $d\om_p=|\gr u|^{p-1}d\h^{n-1}\on_{\del K}$.
This can be further extended to a convex set $K$ of empty interior, so that $K=\del K$. Given a neighborhood $N$ of $K$ and a positive weak solution $u\in W^{1,p}(N)$ of $\lap_pu=0$ in $N\mns K$ with $u=0$ on $K$, there exists a nonnegative Radon measure $\om_p$ on $K$ such that
$$ \int_N |\gr u|^{p-2}\inp{\gr u}{\gr \phi}\dx = -\int_{K} \phi \, d\om_p ,$$
for any $\phi\in C^\infty_0(N)$. As above, $\om_p$ is called the $p$-harmonic measure associated to $u$. We refer to
\cite{LNld} for further details on quasi-linear equations on low dimensional sets.

\subsection{Minkowski problem}\label{subsec:minkprob}
For $n\geq 2$, we consider convex domains $\Om\subset \R^n$ and the Gauss map $\g_\Om :\del\Om\to \Snn$ where $\g_\Om(x)$ is the outer unit normal at $x\in \del\Om$, defined almost everywhere 
on the boundary since $\del\Om$ is locally Lipschitz. Given a positive finite Borel measure $\mu$ on $\Snn$, the Minkowski problem asks for the existence and uniqueness (up to translation when $\Om\mapsto\mu_\Om$ is translation invariant) of a convex body $K=\bar\Om$ with nonempty interior such that $\mu_\Om=\mu$. Here $\mu_\Om=(\g_\Om)_*\eta$ for a measure $\eta$ on $\del\Om$ for which the outer unit normal is unique $\eta$-almost everywhere.
For the density formulas in this discussion, assume $\Om$ has smooth strongly convex boundary and $d\eta=f\,d\h^{n-1}$ with an integrable density $f$. Then the induced spherical measure has density
\begin{equation}\label{eq:msrdensity}
d\mu_\Om( \xi) = \F[h_\Om]( \xi) d \xi
\end{equation}
for a functional $\F$, where $h_\Om$ is the support function. 
For the original Minkowski problem, we have $\eta= \h^{n-1}$ and 
$\mu_\Om= (\g_\Om)_*\h^{n-1}$. The case of the given measure being discrete was considered in \cite{M1, M2}, where the 
corresponding convex domains are polytopes. The continuous case has been shown in \cite{A1, A2,FJ} and the smooth case in \cite{CY}. 
If $\Om$ is of class $C^2_+$ and the covariant derivatives as in \eqref{eq:cov} are defined by the charts as shown in the 
previous subsection,
then recalling \eqref{eq:detkinv} and \eqref{eq:intbd}, we note that 
the density \eqref{eq:msrdensity} is the reciprocal of the Gauss curvature, i.e. 
$$ d\mu_\Om( \xi)=\det(\covtwo  h_\Om + h_\Om \I) d \xi= \frac{d \xi}{\kr(\inv{\g_\Om}( \xi))}, $$
and furthermore, if $\mu_\Om = (\g_\Om)_*\eta$ where $d\eta = f d\h^{n-1}$ for a function $f:\del\Om \to \R$, then we have 
\begin{equation}\label{eq:fhom}
\F[h_\Om]( \xi) = \frac{(f\circ \inv{\g_\Om})( \xi)}{\kr(\inv{\g_\Om}( \xi))} = f(\gr h_\Om)\det(\covtwo  h_\Om + h_\Om \I).
\end{equation}
As examples of the above, the prescribed measure is the harmonic measure $\om$ at origin in \cite{J1} where we have $\mu_\Om = (\g_\Om)_*\om$ and 
$f=-\partial_\nu G=|\gr G|$, for the positive Green's function $G$ with pole at $0$ and the outer normal $\nu$. In the case of $p$-capacitary measures in \cite{CNSXYZ}, we have 
$f= |\gr U|^p$ where $U$ is the $p$-capacitary function. 

In the case of $p$-harmonic measures, $\mu_\Om =  (\g_\Om)_*\om_p$ where 
$\om_p$ is the associated $p$-harmonic measure for a function $u=u_\Om\in W^{1,p}(\Om\cap N)$ under the boundary regularity assumptions used in \eqref{eq:phm}, given by
$d\om_p=|\gr u|^{p-1} d\h^{n-1}\on_{\del\Om}$ where $u$ is $p$-harmonic in $\Om\cap N$ and satisfies 
\begin{equation}\label{eq:omdir}
 \begin{cases}
  \dv (|\gr u|^{p-2}\gr u)=0,\ \ &\text{in}\ \Om\cap N;\\
   u > 0, \ \ &\text{in}\ \Om\cap N;\\
  u= 0, \ \ &\text{on}\ \del\Om,
 \end{cases}
\end{equation}
where $N$ is a neighbourhood of $\del\Om$. Thus, $u\in W^{1,p}(N)$ upon zero extension. It corresponds to the density 
$f= |\gr u|^{p-1}$ in \eqref{eq:fhom} for $1<p<\infty$. Henceforth, we shall fix
\begin{equation}\label{eq:pharmink}
\mu_\Om =  (\g_\Om)_*\om_p, \quad \text{where}\quad 
\om_p(E)= \int_{E\cap \del\Om} |\gr u|^{p-1}d\h^{n-1}. 
\end{equation}

For the classical calculations below, assume a smooth ring with
$\gr u\ne0$ up to its boundary and
$\|u\|_{L^\infty(\bar N\cap \Om)}+\|\gr u\|_{L^\infty(\bar N\cap \Om)}<\infty$.
Using the notations \eqref{eq:pharmink} and 
\eqref{eq:phm}, we have
\[\mu_\Om=(\g_\Om)_*(|\gr u|^{p-1}dS),\qquad
d\mu_\Om(\xi)=\F[h_\Om](\xi)\,d\xi.\]
Writing $F_\Om(\xi):=\inv{\g_\Om}(\xi)=\gr h_\Om(\xi)$, we obtain from \eqref{eq:fhom}
$$  \F[h_\Om]( \xi) = |\gr u(F_\Om( \xi))|^{p-1} \det(\covtwo  h_\Om + h_\Om \I). $$
If $K=\bar\Om$ for a convex domain $\Om$ then we define $\mu_K=\mu_\Om$.

\section{The Brunn-Minkowski inequality}\label{sec:bmineq}

Let $n\geq2$ and let $\mathscr K_0^n$ denote convex bodies containing the origin in their interiors. We fix smooth strongly convex bodies $0\in\mathring L_0\Subset\mathring K_0$ and $\eps_0>0$. Take $N=\R^n\setminus L_0$, so that $\mathring K_0\cap N$ is the convex ring, and let $u_0$ solve
\begin{equation}\label{eq:k0dir}
 \begin{cases}
  \dv (|\gr u_0|^{p-2}\gr u_0)=0,\ \ &\text{in}\ \mathring K_0\cap N;\\
   u_0 > 0, \ \ &\text{in}\ \mathring K_0\setminus L_0;\\
  u_0=0\ \text{on}\ \del K_0,\qquad u_0=\eps_0\ \text{on}\ \del L_0.
 \end{cases}
\end{equation}
For any smooth strongly convex bodies $L\Subset\mathring K$ with $0\in\mathring L$ and any $1<p<\infty$, the solution $u_{K,L}$ of the convex-ring Dirichlet problem with values $0$ on $\del K$ and $\eps_0>0$ on $\del L$ automatically enjoys strong regularity and geometric properties. By Theorem~\ref{thm:plap}(2) and~\cite{Li}, $u_{K,L}\in C^{1,\beta}(\overline{\mathring K\setminus L})$ for some $\beta>0$. By Lewis~\cite{Lewis77}, the superlevel sets $\{u_{K,L}>t\}$ ($0\leq t<\eps_0$) are convex and $|\gr u_{K,L}|>0$ in $\mathring K\setminus L$. Hopf's lemma for the $p$-Laplacian~\cite{Vazquez84}, applied to $u_{K,L}$ and $\eps_0-u_{K,L}$, gives $|\gr u_{K,L}|>0$ on $\del K\cup\del L$. Compactness bounds $|\gr u_{K,L}|$ above and away from zero on $\overline{\mathring K\setminus L}$. In the open ring the solution is therefore smooth and satisfies $\sum_{i,j}a_{ij}\partial_{ij}u_{K,L}=0$, where
\[
a_{ij}=\delta_{ij}+(p-2)\frac{\partial_i u_{K,L}\,\partial_j u_{K,L}}{|\gr u_{K,L}|^2}.
\]
These coefficients are uniformly elliptic and belong to $C^{0,\beta}$ on the closed ring. For these coefficients and the given smooth boundary values, the linear Dirichlet problem has a $C^{2,\beta}$ solution by~\cite[Theorem~6.14]{GT}. The maximum principle identifies it with $u_{K,L}$. Successive boundary Schauder estimates~\cite[Theorem~6.19]{GT} then give $u_{K,L}\in C^\infty(\overline{\mathring K\setminus L})$.
By the Gabriel--Lewis theorem in the form stated by Ma--Ou--Zhang~\cite[Theorem~1.1, Corollary~1.3 and Remark~4.2]{Ma-Ou-Zh}, every level surface $\{u_{K,L}=t\}$ ($0<t<\eps_0$) has strictly positive Gaussian curvature, so all level bodies $\overline{\Om_t(K)}$ are smooth and strongly convex. For the fixed-inner-body problem with $L=L_0$, we denote by $\mathcal N(L_0)$ the class of all smooth strongly convex bodies $K$ satisfying $L_0\Subset\mathring K$. We set $u_{K_0,L_0}=u_0$ and define $u_{K,L_0}$ by the Dirichlet problem with fixed inner body $L_0$
\begin{equation}\label{eq:kdir}
 \begin{cases}
  \dv \big(|\gr u_{K,L_0}|^{p-2}\gr u_{K,L_0}\big)=0,\ \ &\text{in}\ \mathring K\cap N;\\
  u_{K,L_0}(x)= 0, \ \  &\forall\ x\in \del K\cap N;\\
  u_{K,L_0}(x)= u_0(x),\ \  &\forall\ x\in \del N\cap \mathring K.
 \end{cases}
\end{equation}
The inner datum is $u_0=\eps_0$ on $\del L_0$. By extending $0$ outside $K$,  it follows that $u_{K,L_0}\in W^{1,p}(N)$. For superlevel sets we also extend $u_{K,L_0}$ by $\eps_0$ on $L_0$.

As in Section \ref{sec:prelim}, we henceforth denote the image of the $p$-harmonic measure $\mu_{K,L_0}$ on $\Snn$ associated to $u_{K,L_0}$ of \eqref{eq:kdir} as 
\begin{equation}\label{eq:mukint}
\mu_{K,L_0}(E) := \int_{\inv{\g_K}(E)} |\gr u_{K,L_0}|^{p-1} d\h^{n-1}, \quad \text{for any measurable}\ E\sub \Snn,
\end{equation}
where $\g_K$ is the Gauss map.

In other words, for the $p$-harmonic  measure $d\om_{K,L_0}= |\gr u_{K,L_0}|^{p-1} d\h^{n-1}$ on $\del K$, we have 
$\mu_{K,L_0}=(\g_K)_*\om_{K,L_0}$. The functional with fixed inner body $L_0$, distinct from \eqref{eq:bmfunc}, is
\begin{equation}\label{eq:defT}
\T(K;L_0)=\int_{\del K} (h_K \circ \g_K)| \gr u_{K,L_0}|^{p-1}\, d\h^{n-1} 
=\int_{\Snn} h_K ( \xi)\, d \mu_{K,L_0}( \xi).
\end{equation}
\subsection{A homogeneous functional}

\begin{definition}\label{def:hatselection}
For a smooth strongly convex body $K$ containing the origin, set
\[
s(K)=\frac{\displaystyle\int_{\Snn}h_K\,d\xi}{\displaystyle\int_{\Snn}h_{K_0}\,d\xi},
\qquad L(K)=s(K)L_0.
\]
Let $\mathcal U$ consist of such bodies with $L(K)\Subset\mathring K$, and let $u_K$ solve
\begin{equation}\label{eq:hatdir}
\begin{cases}
\Delta_pu_K=0&\text{in }\mathring K\setminus L(K),\\
u_K=0&\text{on }\del K,\\
u_K=\eps_0&\text{on }\del L(K).
\end{cases}
\end{equation}
Thus $u_K=u_{K,L(K)}$. Write $\mu_K=(\g_K)_*(|\gr u_K|^{p-1}dS)$ and $\T(K)=\int h_K\,d\mu_K=\T(K;L(K))$. The density $R_K=d\mu_K/d\xi$ is positive and continuous in the smooth setting considered here. At $K_0$, this selection agrees with $u_0$ and its measure.
\end{definition}
\begin{lemma}\label{prop:hathom}
The domain $\mathcal U$ is stable under positive dilations and Minkowski combinations. For $a>0$,
\[
u_{aK}(ax)=u_K(x),\qquad
\mu_{aK}=a^{n-p}\mu_K,\qquad
\T(aK)=a^{q}\T(K),\quad q=n-p+1.
\]
\end{lemma}
\begin{proof}
Linearity of support functions gives $s(aK)=as(K)$ and
$L((1-\lambda)K_1+\lambda K_2)=(1-\lambda)L(K_1)+\lambda L(K_2)$.
Strict inner inclusion is preserved under these operations. Uniqueness of the Dirichlet problem gives the first identity, and the gradient, area, and support scaling factors give the remaining identities.
\end{proof}

For the following level-set arguments, we suppress the fixed inner body $L$ from the notation: $u_K=u_{K,L}$, $\mu_K=\mu_{K,L}$, $R_K=R_{K,L}$, $\om_K=\om_{K,L}$ and $\T(K)=\T(K;L)$. In the fixed-inner-body comparisons, $L=L_0$.

A key step is to show that supremal convolution of $u_{K_1}$
and $u_{K_2}$ is a subsolution of the $p$-Laplacian in the convex ring associated to $(1-\lambda)K_{1} + \lambda K_{2}$. Towards this, we need to obtain the $p$-Laplacian in terms of the support functions of superlevel sets, which is essential to figure the behavior of $p$-harmonicity with respect to the Minkowski addition of convex sets. This is the content of the following subsection. It follows along the direction of \cite{CS} and similar others.
\subsection{Support functions on \texorpdfstring{superlevel}{superlevel} sets}\label{subsec:supsubl}

For this subsection, let $L\Subset\Om$ be smooth strongly convex bodies with $0\in\mathring L$, and choose $N$ so that $\Om\cap N=\Om\setminus L$. For a $p$-harmonic solution $u$ of the convex-ring Dirichlet problem with values $0$ on $\del\Om$ and $\eps_0>0$ on $\del L$, as noted above, $u\in C^\infty(\overline{\Om\setminus L})$, $|\gr u|>0$ on $\overline{\Om\setminus L}$, $0<u<\eps_0$ in $\Om\setminus L$, and every level body $\overline{\Om_t}$ ($0\leq t<\eps_0$) is smooth and strongly convex. For local differential identities such as Proposition~\ref{prop:plaphu} that are also applied to non-solutions (such as the supremal convolution $u_\lambda^*$), it suffices that $u$ is $C^2$ on $\Om\setminus L$ with $|\gr u|>0$ and $C^2_+$ level sets. Define the continuous extension
\begin{equation}\label{eq:extendedu}
\bar u(x)=\begin{cases}
\eps_0,&x\in L,\\
u(x),&x\in\Om\setminus L,\\
0,&x\notin\Om.
\end{cases}
\end{equation}
We suppress the bar when using level sets and convolutions. For the ring solutions these level bodies, including the outer body, have positive curvature. We require the same property when applying the identities to other functions. Set
\begin{equation}\label{eq:omt}
\Om_t := \set{x\in\R^n}{\bar u(x)>t}, \qquad 0\leq t<\eps_0.
\end{equation}
The sets $\Om_t$ are open convex domains. Their closures are the convex bodies used in support functions. We have $\Om_0=\Om$, and for $0<t<\eps_0$ their boundaries are $\{u=t\}$ in the ring, with outer normal $-\gr u/|\gr u|$. The inverse Gauss maps exist, and

\begin{equation}\label{eq:omdomt}
   % \bar\Om = \bigcup_{0\leq t<\infty} \del \Om_t \quad\text{and}\quad 
    \Om\cap N\ =\underset{0<t<\|u\|_{L^\infty(N\cap \Om)}}{\bigcup} \del \Om_t,
\end{equation}
with each $\del\Om_t$ being diffeomorphic to $\Snn$ via $\Theta_u:\Om\cap N\to \Snn\times(0,\eps_0)$, given by 
\begin{equation}\label{eq:bij}
\Theta_u(x)=(-\gr u(x)/|\gr u(x)|, u(x))\quad\text{and}\quad 
\inv{\Theta_u}(\xi, t)= \inv{\g_{\Om_t}}(\xi). 
\end{equation}
We denote the support function corresponding to $\{u>t\}$ as $h_u: \R^n\times[0,\eps_0)\to\R$ as
\begin{equation}\label{eq:hu}
 h_{u}(y, t) := h_{\Om_{t}} (y),
\end{equation}
with $h_{\Om_{t}} $ as in \eqref{eq:suppf} 
and for derivatives of $h_u$, the notations $\gr$ and $D^2$ shall be used as usual for the first coordinates and $\del_t, \del_t^2$, etc. for the last coordinate. 

The $C^3$ regularity of $u$, its nonvanishing gradient and the positive curvature of the level bodies imply that the map $\Theta_u$ in \eqref{eq:bij} has a $C^2$ inverse, also up to $t=0$ from the ring. Indeed, the normal derivative of $u$ is nonzero, and the Gauss map on each level surface has invertible tangential derivative. Uniqueness of the level and of its point with a given normal makes these local inverses agree. Thus $h_u(\xi,t)=\inp{\xi}{\Theta_u^{-1}(\xi,t)}$ is jointly $C^2$, and compactness gives $h_u(\cdot,t)\to h_\Om$ in $C^2(\Snn)$ as $t\downarrow0$. The same conclusion holds at every interior level.

Now, the goal is to express $\lap_p u$ entirely in terms of $h_u$ and its derivatives. Towards this, we require the following identities. 
 \begin{lemma}\label{lem:ids1}
 Given $\Om$ and $u\in C^2(\Om\cap N)$ as above, we have
 $h_{u}\in C^{2}\big((\R^{n}\setminus\{0\})\times (0,\|u\|_{L^{\infty}})\big)$. If $u$ is also $C^2$ up to $\del\Om$ with nonvanishing gradient there, this regularity and the identities below extend to $t=0$. The following holds at any $x\in\Omega\cap N$:
    \begin{enumerate}
        \item $h_{u}\left( -\gr u/|\gr u|, u\right) = \inp {x}{-\gr u/|\gr u|}$;
        \item $\nabla h_{u}\left( -\gr u/|\gr u|, u\right) = x, \quad \text{ and } \quad 
        D^{2}h_{u} \left( -\gr u/|\gr u|, u\right) \gr u/|\gr u| = 0$;
        \item $\del_{t}h_{u}\left( -\gr u/|\gr u|, u\right) = -1/|\gr u|.$
    \end{enumerate}
\end{lemma}
\begin{proof}
At $x\in\del\Om_{t}$, using $\g_{\Om_t}(x)= -\gr u(x)/|\gr u(x)|$, \eqref{eq:hu} and 
recalling \eqref{eq:supg}, we have 
    \begin{align*}
        h_{u}\left(-\gr u/|\gr u|, u\right) = h_{u}(\g_{\Om_{t}}(x), t) 
        = h_{\Om_{t}}(\g_{\Om_{t}}(x)) 
        = \inp{x}{\g_{\Om_{t}}(x)} 
        = \inp {x}{-\gr u(x)/|\gr u(x)|}, 
    \end{align*}
   for any $ t\in (0, \|u\|_{L^{\infty}})$. This, together with \eqref{eq:omdomt},  concludes $(1)$. Similarly, $(2)$ follows from \eqref{eq:hgrh} and \eqref{eq:hgradh} in the same way using $\nabla h_{\Om_{t}}(\xi) = \g^{-1}_{\Om_{t}}(\xi)$ and $D^{2}h_{\Om_{t}}(\xi)\xi = 0$ for $\xi=-\gr u/|\gr u|$.
To prove $(3)$, first note that using homogeneity \eqref{eq:homgradhom} on $(1)$ and $(2)$, we have
    \begin{equation}\label{eq:12hom}
        h_{u}(-\nabla u,u) = \langle x, -\nabla u\rangle \quad \text{ and } \quad \nabla h_{u}(-\nabla u,u) = x.
    \end{equation}
Differentiating the first identity of \eqref{eq:12hom}, the full gradients of both sides yield
    \begin{align*}
        -D^{2}u\,\nabla h_{u}(-\nabla u,u) + \partial_{t}h_{u}(-\nabla u,u) \nabla u =  - (D^{2}u)x-\gr u,
    \end{align*}
  and the second identity of \eqref{eq:12hom} cancels the first terms of the above, leading to 
    \begin{equation*}
        (\del_{t}h_{u}(-\nabla u,u)+1)\nabla u = 0.
    \end{equation*}
Since $|\gr u|>0$ in $\Om\cap N$, it implies $\del_{t}h_{u}(-\nabla u,u) = -1$
which is enough to conclude $(3)$ as $\del_{t}$ does not change the homogeneity of the first variable of $h_{u}$. This completes the proof. 
\end{proof}
Similarly as \eqref{eq:hu}, we denote the Weingarten map and its inverse corresponding to $u$ as 
\begin{equation}\label{eq:wu}
\Wg_u(x,t) :=\Wg_{\Om_t} (x), \quad\text{and}\quad 
\inv{\Wg_u}(y,t) :=\inv{\Wg_{\Om_t}} (y), 
\end{equation}
for $0\leq t<\eps_0$. As in Section \ref{sec:prelim}, we choose an orthonormal basis $\{e^1,\ldots,e^{n-1}, -\gr u/|\gr u|\}$ where $\{e^i : i\in\{1,\ldots, n-1\}\}$ span the tangent space $T_{-\gr u/|\gr u|} (\Snn)$. As in Section \ref{sec:prelim}, we restrict to $h_u: \Snn\times[0,\eps_0)\to\R$ and 
find the identities on second derivatives of $h_u$ with respect to the covariant derivatives as \eqref{eq:cov} that yield \eqref{eq:grhcov}. %We have the following. 
\begin{lemma}\label{lem:ids2}
Let $\Om$ and $u\in C^2(\Om\cap N)$ be as above, 
 and for $x\in\Om\cap N$ let $\{e^1,\ldots, e^{n-1}\}$ be an orthonormal basis of $T_\xi(\Snn)$, where $\xi=-\gr u(x)/|\gr u(x)|$, with the covariant derivatives taken at $\xi$ as in \eqref{eq:cov}. Then the following holds at $x$:
    \begin{enumerate}
        \item $D^2 h_u\left(-\gr u/|\gr u|, u\right) = \sum_{i,j=1}^{n-1} 
        \inp{\inv{\mathcal W_u}\left(-\gr u/|\gr u|, u\right) e^j}{e^i} e^i \otimes e^j $;
        \item $\cov h_u\left(-\gr u/|\gr u|, u\right)=x+h_u\left(-\gr u/|\gr u|, u\right)\gr u/|\gr u| $;
        \item $\covtwo  h_u \left(-\gr u/|\gr u|, u\right)=\inv{\mathcal W_u}\left(-\gr u/|\gr u|, u\right)- h_u\left(-\gr u/|\gr u|,u\right)\I $.
    \end{enumerate}
\end{lemma}
\begin{proof}
Similarly as in the proof of the above lemma, from \eqref{eq:d2wing},\eqref{eq:wu} and \eqref{eq:omdomt}, $(1)$ follows directly. 
For the rest, note that from \eqref{eq:grhcov}, \eqref{eq:hu} and \eqref{eq:wu}, we have 
\begin{align*}
\gr h_u \left(-\gr u/|\gr u|, u\right)&= -h_u\left(-\gr u/|\gr u|, u\right)\gr u/|\gr u| 
        + \cov h_u\left(-\gr u/|\gr u|, u\right);\\
  \inv{\mathcal W_u}\left(-\gr u/|\gr u|, u\right) &= h_u\left(-\gr u/|\gr u|,u\right)\I +\covtwo  h_u \left(-\gr u/|\gr u|, u\right), 
\end{align*}
and use Lemma \ref{lem:ids1} and \eqref{eq:omdomt} to conclude $(2)$ and $(3)$ 
and complete the proof. 
\end{proof}
We now express the $p$-Laplacian in terms of the support function. We use this formula to construct a subsolution in Subsection~\ref{subsec:sup}.
\begin{proposition} \label{prop:plaphu}
    Given $\Om, N$ and $u\in C^2(\Om\cap N)$ as above, we have 
    
    \[
    M_u:=\covtwo h_u+h_u\I,
    \qquad z_u:=\cov\del_t h_u, 
    \]
    \begin{align*}
        \Delta_{p}u
        = \frac{-1}{(-\del_{t}h_{u})^{p-1}}
        \Bigg[\Tr(M_u^{-1}) +\frac{p-1}{(\del_{t}h_{u})^{2}}
        \left\{\inp{M_u^{-1}z_u}{z_u}-\del_{t}^2h_{u}\right\}\Bigg].
    \end{align*}
    
with $h_u$ evaluated at $\big(-\gr u(x)/|\gr u(x)|, u(x)\big)\in \Snn\times (0,\|u\|_{L^\infty})$ for any $x\in\Om\cap N$. 
\end{proposition}

\begin{proof}
We evaluate $u$ and its derivatives at every $\del\Om_t$ and 
make use of the identities of the above lemmas to obtain relations that hold identically for each $t\in (0,\|u\|_{L^\infty})$, thereby conclude them on $\Om\cap N$ from 
\eqref{eq:omdomt}, similarly as in the proof of the above lemmas. 

First, with respect to the orthonormal basis $\{e^1,\ldots,e^{n-1}, -\gr u/|\gr u|\}$ as in 
Lemma \ref{lem:ids2},  
we decompose the trace to obtain 
\begin{align*}
            \lap u = \Tr(D^{2}u) 
                    &= \sum_{i=1}^{n-1}\inp{D^{2}u\, e^{i}}{e^{i}} 
                    + \inp{D^{2}u\, \pnu}{\pnu} \\
                    &= -|\nabla u|\Tr(\Wg_{u}) + \inp{D^{2}u\, \pnu}{\pnu},
     \end{align*}
where the last equality follows from the first identity in (1) of Lemma~\ref{lem:grad2u}, applied to $\Om_{u(x)}$ at $x$, and \eqref{eq:wu}.
Using the above together with the fact that $\gr (|\gr u|)=D^2u\, \pnu$, 
we obtain the $p$-Laplacian as 
\begin{equation}\label{eq:plap1}
\begin{aligned}
            \lap_{p}u &= |\nabla u|^{p-2}\Big[\lap u + (p-2)\inp{D^{2}u\, \gr u}{\gr u}/|\gr u|^2\Big] \\
            &= |\nabla u|^{p-2}\Big[-|\nabla u|\Tr(\Wg_{u}) + (p-1)\inp{D^{2}u\, \pnu}{\pnu}\Big] \\
            &= -|\nabla u|^{p-1} \Big[\Tr(\Wg_{u}) - \frac{(p-1)}{|\gr u|}
            \inp{\gr (|\gr u|)}{\pnu}\Big] .
        \end{aligned}
\end{equation}
Now, recalling $(3)$ of Lemma \ref{lem:ids1}, we shall use 
$|\gr u|=-1/\del_{t}h_{u}\left( -\gr u/|\gr u|, u\right)$ along with the other identities. To evaluate the last term of \eqref{eq:plap1} in terms of $h_u$, we note that 
\begin{align*}
            \nabla(|\nabla u|) &= \nabla\Big(-1/\del_{t}h_{u}\left( -\gr u/|\gr u|, u\right)\Big)= \frac{1}{(\del_{t}h_{u})^{2}}\nabla\Big(
            \del_{t}h_{u}\left( -\gr u/|\gr u|, u\right)\Big)\\
            &= \frac{1}{(\del_{t}h_{u})^{2}}\Big[-D\left(\pnu\right)\nabla
            \del_{t}h_{u}\left( -\gr u/|\gr u|, u\right) + \del_{t}^2h_{u}\left( -\gr u/|\gr u|, u\right)\nabla u\Big].
        \end{align*}
Taking the inner product with $\nabla u/|\nabla u|$ on the above, we obtain
\begin{equation*}
\begin{aligned}
            \inp{\nabla(|\nabla u|)}{\pnu} &= \frac{1}{(\del_{t}h_{u})^{2}} \left[-\inp{D\left(\pnu\right)\nabla\del_{t}h_{u}}{\pnu} + |\nabla u|\del_{t}^{2}h_{u}\right] \notag\\
            &= \frac{1}{(\del_{t}h_{u})^{3}}\left[-(\del_{t}h_{u})\inp{D\left(\pnu\right)\nabla\del_{t}h_{u}}{\pnu} - \del_{t}^{2}h_{u}\right],
        \end{aligned} 
\end{equation*}
 which is applied to \eqref{eq:plap1} along with using $\del_t h_u =-1/|\gr u|$, to obtain
 \begin{equation*}%\label{eq:plap2}
\lap_p u= \frac{-1}{(-\del_{t}h_{u})^{p-1}}\Big[\Tr(\Wg_{u}) +\frac{(p-1)}{(\del_{t}h_{u})^{2}}\left[-(\del_{t}h_{u})\inp{D\left(\pnu\right)\nabla\del_{t}h_{u}}{\pnu} - \del_{t}^{2}h_{u}\right]\Big].
\end{equation*}
To complete the proof, we need to look into the term in the right hand side of the above containing the Jacobian of $\gr u/|\gr u|$, which can be computed as 
\begin{equation*}
            D\left(\pnu\right) = \frac{1}{|\nabla u|}\Big(D^{2}u 
            -\frac{1}{|\gr u|^2} D^{2}u\,\gr u\otimes \gr u\Big).
        \end{equation*}
Notice that $D\left(\pnu\right)\gr u=0$. Furthermore, $\del_t h_u(\cdot,t)$ is positively homogeneous of degree $1$ and of class $C^1$ on $\R^n\mns\{0\}$, so (1) of Lemma~\ref{lem:sphder} gives
\begin{equation}\label{eq:grdthu}
\gr \del_t h_u \left(-\gr u/|\gr u|, u\right)= -\del_t h_u\left(-\gr u/|\gr u|, u\right)\gr u/|\gr u| 
        + \cov \del_t h_u\left(-\gr u/|\gr u|, u\right).
\end{equation}
Using these, we have that 
\begin{equation}\label{eq:j1}
\begin{aligned}
\inp{D\left(\pnu\right)\nabla\del_{t}h_{u}}{\pnu}
&=\inp{D\left(\pnu\right)\cov\del_{t}h_{u}}{\pnu}\\
&=\inp{D\left(\pnu\right)^T\pnu}{\cov\del_t h_u}.
\end{aligned}
\end{equation}
We eliminate the last term by computing the derivative of 
$\nabla h_{u}\left( -\gr u/|\gr u|, u\right) = x$, from $(2)$ of Lemma \ref{lem:ids1}. On its $j$-th component i.e. $\del_jh_u(-\pnu,u)=x_j$, differentiating with $\del_i$ yields 
\begin{align*}
            \delta_{i,j} &= \partial_{i}\Big(\partial_{j}h_{u}(-\pnu, u)\Big) \\
            &= \sum_{k}\del_{k}\del_{j}h_{u}(-\pnu, u)\del_{i}\left(\frac{-\del_{k}u}{|\nabla u|}\right) + \del_{t}\del_{j}h_{u}(-\pnu, u)\del_{i}u.
        \end{align*}
In matrix form, the above reads as
        \begin{equation*}
            D\left(-\pnu\right)D^{2}h_{u}(-\pnu, u) + \nabla u\otimes\nabla\del_{t}h_{u}(-\pnu, u) = \I,
        \end{equation*}
which, upon transposition, leads to 
\begin{equation}\label{eq:mj}
            D^{2}h_{u}(-\pnu, u)D\left(-\pnu\right)^T + \nabla\del_{t}h_{u}(-\pnu, u)\otimes\gr u = \I.
        \end{equation}
To generate the rightmost term of \eqref{eq:j1}, first we multiply \eqref{eq:mj} with $\pnu$ to get 
$$-D^2h_u(-\pnu, u) D(\pnu)^T\pnu+|\gr u|\gr\del_t h_u(-\pnu, u) = \pnu,$$
and then take inner product with respect to any $e^j$ to obtain 
$$ \inp{D^2h_u(-\pnu, u) D(\pnu)^T\pnu}{e^j}
=|\gr u|\inp{\gr\del_t h_u(-\pnu, u)}{e^j}.$$
Now, from $(1)$ of Lemma \ref{lem:ids2} and \eqref{eq:basis}, it is easy to see that  
$D^2h_u(-\pnu, u)e^j=\inv{\Wg_u} e^j$ and recalling $(2)$ of Lemma \ref{lem:ids1}, we have $D^2h_u(-\pnu, u)\pnu =0$. Therefore, 
taking summation over $j\in\{1,\ldots, n-1\}$ of the above and using from \eqref{eq:basis} and \eqref{eq:grdthu}, 
we have $$\inv{\Wg_u} D(\pnu)^T\pnu = |\gr u|\cov\del_t h_u.$$
Thus, we have obtained
\begin{equation}\label{eq:j2}
D(\pnu)^T\pnu = |\gr u|\Wg_u\cov\del_t h_u. 
\end{equation}
The $p$-Laplacian formula above, together with \eqref{eq:j1} and \eqref{eq:j2}, can be re-written as
$$ \lap_p u= \frac{-1}{(-\del_{t}h_{u})^{p-1}}\Big[\Tr(\Wg_{u}) +\frac{(p-1)}{(\del_{t}h_{u})^{2}}\left[\inp{\Wg_u\cov\del_t h_u}{\cov\del_t h_u} - \del_{t}^{2}h_{u}\right]\Big],$$
Finally, $\Wg_u=M_u^{-1}$ and $\cov\del_t h_u=z_u$ by Lemma~\ref{lem:ids2} and the definitions, which gives the stated formula. 
\end{proof}
%
%The proof of Proposition~\ref{prop:plaphu} is local. At $x$ it uses only that, for every $y$ near $x$, $u$ is of class $C^2$ on an open set containing $\{u=u(y)\}$, with $\gr u\neq0$ there, and $\{u=u(y)\}=\del\Om_{u(y)}$ for a convex domain $\Om_{u(y)}$ of class $C^2_+$, and that $h_u$ is of class $C^2$ near $\Theta_u(x)$. For the supremal convolution $u^*_\lambda$ in Proposition~\ref{prop:subsoln}, these conditions hold on the whole open ring. Its level bodies have the support functions $h^*_\lambda=(1-\lambda)h_{u_1}+\lambda h_{u_2}$, with $\covtwo h^*_\lambda+h^*_\lambda\I=M_\lambda>0$, so they are of class $C^2_+$. Since $\gr h^*_\lambda$ gives coordinates on the ring, $u^*_\lambda$ is of class $C^2$ there, with $|\gr u^*_\lambda|=-1/\del_t h^*_\lambda>0$.
%

\subsection{A Limiting Characterization}\label{subsec:limchar}
We maintain the setup of the previous subsection with convex domains $\Om$ and 
$\Om_t$ with respect to a $p$-harmonic function $u$ vanishing on $\del\Om$. As above, $u=\eps_0$ on the inner boundary $\del N\cap\Om$. Thus
\begin{equation}\label{eq:defTom}
\T (\Om)=\int_{\del \Om} (h_\Om \circ \g_\Om)| \gr u|^{p-1}\, d\h^{n-1} ,
\end{equation}
is consistent with \eqref{eq:defT}, if $u=u_\Om$ solves the Dirichlet problem
\begin{equation}\label{eq:omdir1}
 \begin{cases}
  \dv \big(|\gr u|^{p-2}\gr u\big)=0,\ \ &\text{in}\ \Om\cap N;\\
  u(x)= 0, \ \  &\forall\ x\in \del \Om\cap N;\\
  u(x)= \eps_0,\ \  &\forall\ x\in \del N\cap \Om.
 \end{cases}
\end{equation}
Towards a limiting characterization, first we show the following.
\begin{lemma}\label{lem:fluxdist}
Let $\Omega_{t}=\{u>t\}$ be convex and $C^2_+$ as in \eqref{eq:omt}
for every $0 \leq t< \eps_0$, with $u$ being the solution of \eqref{eq:omdir1} for some $\eps_0>0$. Then, for any $0<s<\eps_0$ and $x\in\del\Om_s$, we have
\begin{equation}\label{eq:decrt0}
\mathrm{dist}(x,\partial \Om)
\leq s \|1/|\gr u|\|_{L^\infty(\del \Om)}.
\end{equation}
\end{lemma}
\begin{proof}
To estimate the change in support functions, Proposition~\ref{prop:plaphu} and $\lap_pu=0$ give
\[
\del_t^2h_u
=\frac{(\del_th_u)^2}{p-1}\Tr(M_u^{-1})
+\inp{M_u^{-1}\cov\del_th_u}{\cov\del_th_u}>0.
\]
Since $\del_th_u<0$, its absolute value decreases as $t$ increases.  By integrating in $t$ and using \eqref{eq:defdh} and \eqref{eq:hdh}, we obtain
\begin{equation}\label{eq:omstok0}
\begin{aligned}
\dist(x, \del \Om) &\leq d_\h(\del\Om_s, \del \Om)= d_\h(\Om_s, \Om)= 
\|h_{\Om_s} -h_\Om\|_{L^\infty(\Snn)} \\
&=\|h_{u}(\cdot, s) -h_{u}(\cdot, 0)\|_{L^\infty(\Snn)}\leq s \|\del_t h_{u}(\cdot, 0)\|_{L^\infty(\Snn)}, 
\end{aligned}
\end{equation}
where $h_{u}$ is as in \eqref{eq:hu}. 
Then \eqref{eq:decrt0} follows from (3) of Lemma \ref{lem:ids1}. The proof is finished. 
\end{proof}
In a sufficiently thin tubular neighborhood of the smooth boundary $\del\Om$, every interior point $x$ has a unique nearest projection $p_{\del\Om}(x)\in\del\Om$. By \eqref{eq:decrt0}, all sufficiently small positive level surfaces lie in this neighborhood. At those points,
\begin{equation}\label{eq:proj0}
    |p_{\partial \Om}(x)-x| = \mathrm{dist}(x,\partial \Om), \quad \text{and}\quad \mathbf{g}_{\Omega}(p_{\partial \Om}(x)) = \frac{p_{\partial \Om}(x)-x}{|p_{\partial \Om}(x)-x|},
\end{equation}
This leads to a way of replacing the gradient, as shown in the following.
\begin{lemma}\label{lem:tech1}
Let $\Omega_{t}=\{u>t\}$ be convex and $C^2_+$ as in \eqref{eq:omt} 
for every $0 \leq t< \eps_0$, with $0\in \Om$ and $u$ being the solution of \eqref{eq:omdir1} for some $\eps_0>0$. Then, for small enough $s>0$ and $x\in \del\Om_s$, 
we have 
\begin{equation}\label{eq:u/d}
 \frac{u(x)}{\mathrm{dist}(x,\partial\Omega)}= |\gr u(p_{\del\Om}(x))| + O(s^{\beta}), 
\end{equation}
for some $\beta\in (0,1)$, where $p_{\del\Om}$ denotes the projection as in \eqref{eq:proj0}.
\end{lemma}
\begin{proof}
 Letting $\xi=\mathbf{g}_{\Omega}(p_{\partial \Om}(x))$ and $d(x)=\dist(x,\del\Om)$, we have $x= p_{\del\Om}(x)-\xi d(x)$ for $x\in \del\Om_s$. For small $s$, the normal segment lies in the closed ring. Let us denote the restriction of $u$ to this segment by $\gamma_x:[0,d(x)]\to\R$, defined as 
$$ \gamma_x(t)= u\big(p_{\del\Om}(x)-t\xi\big),\quad \text{and hence,}\quad 
\gamma_x'(t)= \inp{\gr u\big(p_{\del\Om}(x)-t\xi\big)}{-\xi}.$$
 Note that $\gamma_x(d(x))= u(x)$ and $\gamma_x(0)=u(p_{\del\Om}(x)) =0$, since $u$ vanishes on $\del\Om$. Thus we have
 $$ u(x)= \gamma_x(d(x))-\gamma_x(0)=\int^{d(x)}_0 \gamma_x'(t)\, d t. $$
For any $x\in\del\Om_s$, \eqref{eq:decrt0} gives $d(x)\leq sM$ for some $M>0$. Also, $\gamma_x'(0)=|\gr u(p_{\del\Om}(x))|$ because $\xi=-\gr u(p_{\del\Om}(x))/|\gr u(p_{\del\Om}(x))|$.
The smooth boundary regularity assumption gives a uniform $C^{1,\beta}$ bound on the closed ring. In fact, subtracting $\gamma_x'(0)$ under the integral bounds the error by $C d(x)^\beta\leq C' s^\beta$, uniformly in $x\in\del\Om_s$. This proves \eqref{eq:u/d} for small $s$.
This completes the proof. 
\end{proof}
Let us denote the density of the pushforward $p$-harmonic measure by
\begin{equation}\label{eq:defRK}
R_K(\xi):=
\frac{|\nabla u_K(\nabla h_K(\xi))|^{p-1}}
{\mathcal K_{\partial K}(\nabla h_K(\xi))},
\qquad \xi\in\Snn.
\end{equation}
For a convex domain $\Om$, we denote $R_\Om:=R_{\bar\Om}$. 
Thus we have 
$$d\mu_K(\xi)=R_K\,d\xi, \quad\text{and}\quad \T(K)=\int_{\Snn}h_K(\xi)R_K(\xi)\,d\xi.$$ 
Now we prove the following limiting characterizations. 
\begin{proposition}\label{prop:wlimchar}
Let $\Om$ and $u$ satisfy the smoothness, nonvanishing-gradient and level-set hypotheses of Subsection~\ref{subsec:supsubl}, with $u$ solving \eqref{eq:omdir1} for $\eps_0>0$ and $p>1$. Let 
$\Omega_{t}=\{u>t\}$ as in \eqref{eq:omt} be convex for every $0 \leq t< 
\eps_0$. Then, given $R$ as in \eqref{eq:defRK} and 
$\T$ as in \eqref{eq:defTom}, we have
\begin{align}
\label{eq:rlimchar} & \text{(i)}\quad 
\kr_{\del\Om}(\gr h_\Om(\xi))R_\Om(\xi)= \lim_{s\to 0^{+}} \left(\frac{u(\gr h_u(\xi, s))}{\mathrm{dist}(\gr h_u(\xi, s),\partial\Omega)}\right)^{p-1};\\
\label{eq:wlimchar}& \text{(ii)}\quad\T(\Omega) = \lim_{s\to 0^{+}}\int_{\del\Om_s} 
   h_{\Omega_{s}}(\g_{\Omega_{s}}(x))
   \left(\frac{u(x)}{\mathrm{dist}(x,\partial\Omega)}\right)^{p-1} d\mathcal{H}^{n-1}(x). 
\end{align}
\end{proposition}
\begin{proof}
As earlier, let us denote $d(x)=\dist(x,\del\Om)$ and we shall choose a constant $M>0$ large enough as much as required below, that may vary from line to line but is determined by the given data. 
First, recalling \eqref{eq:hu} and Lemma \ref{lem:ids1}, 
we note that for 
$x=\gr h_u(\xi, s)\in \del\Om_s$ for $s>0$ small enough, we rewrite \eqref{eq:u/d} of Lemma \ref{lem:tech1} as
\begin{equation}\label{eq:u/d0}
 \frac{u(\gr h_u(\xi, s))}{d(\gr h_u(\xi, s))}= |\gr u(p_{\del\Om}(x))| + O(s^{\beta}).
\end{equation}
Now, we show that the point $p_{\del\Om}(x)$ is close enough to the 
point $\gr h_u(\xi, 0)$ on $\del\Om$. We estimate using \eqref{eq:decrt0} in the following,
\begin{equation}\label{eq:estproj}
\begin{aligned}
|p_{\del\Om}(x)-\gr h_u(\xi, 0)| 
&\leq |p_{\del\Om}(x)-x|+|\gr h_u(\xi, s)-\gr h_u(\xi, 0)|\\
&\leq d(x)+s\sup_{\eta\in\Snn,\,0\leq t\leq s}|\del_t\gr h_u(\eta,t)| \leq 2M s.
\end{aligned}
\end{equation}
Therefore, recalling the assumed smoothness of $u$ up to the boundary, we have
\begin{equation}\label{eq:grupest}
|\gr u(p_{\del\Om}(x))-\gr u(\gr h_u(\xi, 0))| \leq \|u\|_{C^{1,\beta}}|p_{\del\Om}(x)-\gr h_u(\xi, 0)|^\beta \leq Ms^\beta,
\end{equation}
where we have used \eqref{eq:estproj} for the latter inequality. 
Thus, using \eqref{eq:grupest}, 
we can rewrite \eqref{eq:u/d0} as 
\begin{equation}\label{eq:u/d'}
 \frac{u(\gr h_u(\xi, s))}{d(\gr h_u(\xi, s))}= |\gr u(\gr h_u(\xi, 0))| + O(s^{\beta}).
\end{equation}
This together with \eqref{eq:defRK}, yields \eqref{eq:rlimchar}. To prove the second one, using \eqref{eq:intbd} and \eqref{eq:detkinv}, we rewrite the integral in the right of \eqref{eq:wlimchar} as
$$ \int_{\del\Om_s} 
   h_{\Omega_{s}}(\g_{\Omega_{s}}(x))
  \left(\frac{u(x)}{d(x)} \right)^{p-1} d\mathcal{H}^{n-1}(x)
   =\int_{\Snn} h_{u}(\xi, s) \left(\frac{u(\gr h_u(\xi, s))}{d(\gr h_u(\xi, s))} \right)^{p-1}\frac{d\xi}{\kr_{\del\Om_s}}, $$
   where $\kr_{\del\Om_s}$ in a spherical integral denotes its pullback $\kr_{\del\Om_s}(\gr h_u(\xi,s))$.  Then, using 
   \eqref{eq:omstok0} and \eqref{eq:u/d'}, we have 
\begin{align*}
 \int_{\del\Om_s} 
   &h_{\Omega_{s}}(\g_{\Omega_{s}}(x))
  \left(\frac{u(x)}{d(x)} \right)^{p-1} d\mathcal{H}^{n-1}(x)\\
   &=\int_{\Snn} h_{u}(\xi, 0)  |\gr u(\gr h_u(\xi, 0))|^{p-1}\frac{d\xi}{\kr_{\del\Om_s}} + O(s)+O(s^\beta) \\
   &= \int_{\Snn} h_{u}(\xi, 0)  |\gr u(\gr h_u(\xi, 0))|^{p-1}\frac{d\xi}{\kr_{\del\Om}} + O(s)+O(s^\beta) + o(1),
\end{align*}
where for the latter equality, we have used the uniform convergence $1/\kr_{\del\Om_s}=1/\kr_{\del\Om}+o(1)$ from the $C^2$ level-support convergence in subsection~\ref{subsec:supsubl} and the smooth assertion of Remark~\ref{thm:c2approx}.  Now, converting the integral on the right hand side of the above using \eqref{eq:intbd} and \eqref{eq:detkinv} again, we obtain
$$ \int_{\del\Om_s} 
   h_{\Omega_{s}}(\g_{\Omega_{s}}(x))
  \left(\frac{u(x)}{d(x)} \right)^{p-1} d\mathcal{H}^{n-1}(x)
  = \int_{\del\Om} h_{\Omega}(\g_{\Omega}(x))
  |\gr u(x)|^{p-1} d\mathcal{H}^{n-1}(x) + o(1).$$
  Taking limit $s\to 0^+$ on the above and recalling \eqref{eq:defTom}, is enough to obtain \eqref{eq:wlimchar}, thereby concluding the proof. 
\end{proof}

\subsection{Sub-solutions from supremal convolution and unconditional boundary bounds}\label{subsec:sup}

We now consider pairs of smooth strongly convex rings $L_1\Subset\mathring K_1$ and $L_2\Subset\mathring K_2$ with $0\in\mathring L_1\cap\mathring L_2$. For $\lambda\in[0,1]$, set
\[
K_\lambda=(1-\lambda)K_1+\lambda K_2,
\qquad
L_\lambda=(1-\lambda)L_1+\lambda L_2.
\]
Note that $L_\lambda\Subset\mathring K_\lambda$ is again a smooth strongly convex ring. This unified setting covers both the fixed-inner-body problem ($L_1=L_2=L_0$, so $L_\lambda=L_0$) and the homogeneous problem of Definition~\ref{def:hatselection} ($L_i=L(K_i)=s(K_i)L_0$, for which Lemma~\ref{prop:hathom} gives $L_\lambda=L(K_\lambda)$). For each smooth strongly convex ring $L\Subset\mathring K$, let $u_{K,L}$ solve
\begin{equation}\label{eq:kdir1}
 \begin{cases}
  \dv \big(|\gr u_{K,L}|^{p-2}\gr u_{K,L}\big)=0,\ \ &\text{in}\ \mathring K\setminus L;\\
  u_{K,L}(x)= 0, \ \  &\forall\ x\in \del K;\\
  u_{K,L}(x)= \eps_0,\ \  &\forall\ x\in \del L,
 \end{cases}
\end{equation}
and extend $u_{K,L}$ by $\eps_0$ on $L$ and $0$ outside $K$ as in \eqref{eq:extendedu}. The maximum principle for the $p$-Laplacian for \eqref{eq:kdir1} leads to $\|u_{K,L}\|_{L^\infty(\mathring K\setminus L)}=\eps_0$.

The convexity and strong convexity of the superlevel sets follow from Lewis~\cite{Lewis77} and Ma--Ou--Zhang~\cite{Ma-Ou-Zh}.
\begin{theorem}\label{thm:conv}
Let $L\Subset\mathring K$ be smooth strongly convex bodies and let $u_{K,L}$ solve \eqref{eq:kdir1} for $1<p<\infty$. With the extension \eqref{eq:extendedu}, the set $\{x\in K:u_{K,L}(x)>t\}$ is convex and smooth with strictly positive boundary Gaussian curvature for every $0\leq t<\eps_0$, and $|\gr u_{K,L}|>0$ on $\overline{\mathring K\setminus L}$.
\end{theorem}

Now we are in a position to use Proposition \ref{prop:plaphu}. Given $u_i=u_{K_i,L_i}$ solving \eqref{eq:kdir1} for $i\in\{1,2\}$ and any $\lambda\in [0,1]$, the supremal convolution of $u_1$ and $u_2$ is given by
\begin{equation}\label{eq:supcon}
u_{\lambda}^*(x) := \sup_{\substack{y,z\in\R^n\\ x = (1- \lambda) y + \lambda z}}\big[\min\{u_{1}(y),u_{2}(z)\}\big].
\end{equation}
This is defined so that for any $0<t\leq\eps_0$ and $\lambda\in [0,1]$, we have 
\begin{equation}\label{eq:conprop}
\{u_{\lambda}^*\geq t\}=(1- \lambda)\{u_{1}\geq t\}+\lambda \{u_{2}\geq t\}. 
\end{equation}
Using the $p$-Laplacian expression in Proposition \ref{prop:plaphu}, we show that $u_{\lambda}^*$ is a sub-solution. We require 
the following lemma, cf.~\cite[pp. 472--473]{CS}. 
\begin{lemma}\label{lem:cstech}
 Given any symmetric and positive definite matrices $M_1,M_2\in\R^{d\times d}$, vectors $z_1,z_2\in\R^d$ (with $d=n-1$ in the application) and $t_1,t_2\in \R$, let 
 $M_{\lambda} = (1-\lambda)M_{1} + \lambda M_{2}$, $z_{\lambda} = (1-\lambda)z_{1} + \lambda z_{2}$ and $t_{\lambda} = (1-\lambda)t_{1} + \lambda t_{2}$ for any $\lambda\in [0,1]$. Then we have the following:
    \begin{align}\label{mtz1}
     \text{(i)}\quad   \langle M_{\lambda}^{-1}z_{\lambda}, z_{\lambda}\rangle &\leq \convexcomb{\inp{M^{-1}_{1}z_{1}}{z_{1}}}{\inp{M^{-1}_{2}z_{2}}{z_{2}}};\\
    \label{mtz2}   \text{(ii)}\quad   t_{\lambda}^{2}\,\mathrm{Tr}(M^{-1}_{\lambda}) &\leq \convexcomb{t_{1}^{2}\,\mathrm{Tr}(M_{1}^{-1})}{t_{2}^{2}\,\mathrm{Tr}(M_{2}^{-1})}.
    \end{align}
For $0<\lambda<1$, equality in \eqref{mtz1} holds if and only if $M_1^{-1}z_1=M_2^{-1}z_2$. Equality in \eqref{mtz2} holds if and only if $t_1M_2=t_2M_1$. At $\lambda=0,1$, equality is automatic.
\end{lemma}

\begin{proof}
The identity
\[
\inp{M^{-1}z}{z}=\max_v\{2\inp{v}{z}-\inp{Mv}{v}\}
\]
proves joint convexity in $(M,z)$. Since the maximizing vector is uniquely $M^{-1}z$, equality for an interior convex combination occurs exactly when the two maximizers agree. Apply this to $z_i=t_i e_j$ for every vector of an orthonormal basis and sum in $j$. Equality of the sum requires $t_1M_1^{-1}=t_2M_2^{-1}$, equivalently $t_1M_2=t_2M_1$, proving the second assertion.
\end{proof}

Using the above Lemma \ref{lem:cstech} and  Proposition \ref{prop:plaphu}, we have the following. 

\begin{proposition}\label{prop:subsoln}
For any smooth strongly convex rings $L_1\Subset\mathring K_1$ and $L_2\Subset\mathring K_2$, and $\lambda\in[0,1]$, let 
$K_\lambda= (1-\lambda)K_1+\lambda K_2$, $L_\lambda=(1-\lambda)L_1+\lambda L_2$, and let
$u_{\lambda}^{*}$ be as in \eqref{eq:supcon}. Then we have $ -\lap_{p}u_{\lambda}^{*} \leq 0$ in $\mathring{K_\lambda}\setminus L_\lambda$.
\end{proposition}
\begin{proof}
 For $i\in \{1,2\}$, let $h_i(\xi, t)=h_{u_i}(\xi, t)$ as in \eqref{eq:hu} with $u_i=u_{K_i,L_i}$ and let us denote 
 $$M_i=\covtwo h_{i}(\xi, t) + h_{i}(\xi, t)\I, \quad t_i= \del_t h_i(\xi, t), \quad z_i=\cov \del_t h_i(\xi, t).$$ 
 Let $h^*_\lambda (\cdot, t) 
 =h_{\{u^*_\lambda\geq t\}}$ for $0<t<\eps_0$, with $u^*_\lambda $ as in \eqref{eq:supcon} so that from \eqref{eq:conprop} and \eqref{eq:suppminkadd}, we have 
 $$h^*_\lambda=(1-\lambda)h_1+\lambda h_2.$$ Hence, the convex combinations as in Lemma \ref{lem:cstech} are given by 
$$M_\lambda =\covtwo h^*_\lambda(\xi, t) + h^*_\lambda(\xi, t)\I, 
\quad t_\lambda= \del_t h^*_\lambda(\xi, t), \quad z_\lambda=\cov \del_t h^*_\lambda(\xi, t),$$ 
for any $\lambda\in [0,1]$. Also, recalling \eqref{eq:bij},  
let $\Theta_i=\Theta_{u_i}$ and 
$\Theta^*_\lambda=\Theta_{u^*_\lambda}$. Then from \eqref{eq:bij} and \eqref{eq:hu}, 
\begin{equation}\label{eq:invlin}
\inv{\Theta^*_\lambda}=\gr h^*_\lambda
=(1-\lambda)\gr h_1+\lambda \gr h_2=(1-\lambda)\inv{\Theta_1}+\lambda \inv{\Theta_2}.
\end{equation}

Every point of $\mathring{K_\lambda}\setminus L_\lambda$ lies on exactly one positive level surface of $u^*_\lambda$. Indeed, $M_\lambda>0$ and $\del_th^*_\lambda<0$, so the map $(\xi,t)\mapsto\gr h^*_\lambda(\xi,t)$ has nonzero normal derivative and invertible tangential derivative. Distinct levels have disjoint boundaries, and the Gauss parametrization is injective on each level. Hence this map is injective and gives smooth coordinates on the open ring. The level bodies converge to $K_\lambda$ as $t\downarrow0$ and to $(1-\lambda)L_1+\lambda L_2=L_\lambda$ as $t\uparrow\eps_0$. Thus $u^*_\lambda$ extends continuously to both boundaries with values $0$ and $\eps_0$, respectively. Consequently, for $x\in\mathring{K_\lambda}\setminus L_\lambda$, if $(\xi,t)=\Theta^*_\lambda(x)$, then \eqref{eq:invlin} gives $x=(1-\lambda)\inv{\Theta_1}(\xi, t)+\lambda \inv{\Theta_2}(\xi, t)$ which implies 
$$\Theta^*_\lambda(x)=(\xi, t)=\Theta_1(x_1)=\Theta_2(x_2),\quad \text{for some}\quad x_1\in\mathring{K_1}\setminus L_1,\ x_2\in\mathring{K_2}\setminus L_2.
$$ 
Thus, we can use Proposition \ref{prop:plaphu} uniformly at 
any $(\xi,t)\in\Snn\times(0,\eps_0)$.
We have, $\lap_{p}u_{1} = 0$ in $\mathring{K_1}\setminus L_1$ and $\lap_{p}u_{2}=0$ in $\mathring{K_2}\setminus L_2$ from \eqref{eq:kdir1}. Hence, using Proposition \ref{prop:plaphu}, 
at any $(\xi,t)\in\Snn\times(0,\eps_0)$, we have 
\begin{equation}\label{plap12}
\begin{aligned}
    t_1^2\,\mathrm{Tr}(M_{1}^{-1}) + (p-1)[\langle M_{1}^{-1}z_{1}, z_{1}\rangle - \partial^{2}_{t}h_{1}] &= 0, \quad\text{in}\ \mathring{K_1}\setminus L_1,\\
    t_2^2\,\mathrm{Tr}(M_{2}^{-1}) + (p-1)[\langle M_{2}^{-1}z_{2}, z_{2}\rangle - \partial^{2}_{t}h_{2}] &= 0, 
    \quad\text{in}\ \mathring{K_2}\setminus L_2. 
    \end{aligned}
\end{equation}
Using Proposition \ref{prop:plaphu}, Lemma \ref{lem:cstech} and \eqref{plap12} respectively, we obtain 
    \begin{align*}
      (-\del_{t}h^*_{\lambda})^{p+1} (-\Delta_{p}u^{*}_{\lambda}) &=t_\lambda^2\,\mathrm{Tr}(M^{-1}_{\lambda}) + (p-1)\left[\inp{M^{-1}_{\lambda}z_{\lambda}}{z_{\lambda}} - \partial_{t}^{2}h^*_{\lambda}\right]\\
        &\leq\convexcomb{t_1^2\,\mathrm{Tr}(M_{1}^{-1})}{t_2^2\,\mathrm{Tr}(M_{2}^{-1})} \\
        &\quad+(p-1)\left[\convexcomb{\inp{M_{1}^{-1}z_{1}}{z_{1}}}{\inp{M_{2}^{-1}z_{2}}{z_{2}}}\right]\\
        &\quad-(p-1)\left[\convexcomb{\partial_t^{2}h_{1}}{\partial_{t}^{2}h_{2}}\right]\\
        &= (1-\lambda)\bigbrackets{t_1^2\,\mathrm{Tr}(M_{1}^{-1}) + (p-1)\left[\inp{M_{1}^{-1}z_{1}}{z_{1}} -\partial_{t}^{2}h_{1}\right]}\\
        &\quad+ \lambda\bigbrackets{t_2^2\,\mathrm{Tr}(M_{2}^{-1}) + (p-1)\left[\inp{M_{2}^{-1}z_{2}}{z_{2}}-\partial_{t}^{2}h_{2}\right]}=0.
    \end{align*}    
Recalling Lemma \ref{lem:ids1}, 
$-\del_{t}h^*_{\lambda}>0$ and the proof is finished. 
\end{proof}
\begin{corollary}\label{cor:subsolncomp}
For any smooth strongly convex rings $L_1\Subset\mathring K_1$ and $L_2\Subset\mathring K_2$, and $\lambda\in[0,1]$, let 
$K_\lambda= (1-\lambda)K_1+\lambda K_2$, $L_\lambda=(1-\lambda)L_1+\lambda L_2$, and let
$u_{\lambda}^{*}$ be as in \eqref{eq:supcon}. Then we have 
$ u_{\lambda}^{*} \leq u_{K_\lambda,L_\lambda}$ in $K_{\lambda}\setminus L_\lambda$ (and on all of $\R^n$ upon extension \eqref{eq:extendedu}).
\end{corollary}
\begin{proof}
The endpoint cases are immediate. For $0<\lambda<1$, Proposition~\ref{prop:subsoln} gives $-\lap_pu^*_\lambda\leq0=-\lap_pu_{K_\lambda,L_\lambda}$ on $\mathring{K_\lambda}\setminus L_\lambda$. On $\del K_\lambda$, a decomposition with positive minimum in \eqref{eq:supcon} would have $y\in\mathring{K_1}$ and $z\in\mathring{K_2}$, and hence $x\in\mathring{K_\lambda}$, a contradiction. Thus both functions have outer trace zero. On $\del L_\lambda$, the extension satisfies $u_i\leq\eps_0$, so $u^*_\lambda\leq\eps_0=u_{K_\lambda,L_\lambda}$. Set $M=\max_{i=1,2}\|\gr u_i\|_{L^\infty(\mathring K_i\setminus L_i)}<\infty$. At corresponding level points, Lemma~\ref{lem:ids1} gives
\[
-\del_t h_\lambda^*=\frac{1-\lambda}{|\gr u_1|}+\frac{\lambda}{|\gr u_2|}\geq M^{-1},
\qquad |\gr u_\lambda^*|\leq M.
\]
Thus $u_\lambda^*\in W^{1,\infty}(\mathring K_\lambda\setminus L_\lambda)\subset W^{1,p}(\mathring K_\lambda\setminus L_\lambda)$, and its continuous boundary values give the required traces. The comparison principle on the ring now proves the claim. Both extended functions equal $\eps_0$ on $L_\lambda$ and zero outside $K_\lambda$, so the comparison holds everywhere on $\R^n$.
\end{proof}

For $p=2$ the convex-ring comparison goes back to Borell~\cite{B2}. For general $p$, the supremal-convolution comparison and the reciprocal-gradient formula are given by Bianchini--Salani~\cite[Section~5]{BS}. We use this comparison in Proposition~\ref{prop:subsoln} and Corollary~\ref{cor:subsolncomp} to obtain the following pointwise inequalities for the boundary gradient and spherical density on $\Snn$, valid both for the fixed-inner-body problem ($L_1=L_2=L_0$) and for the homogeneous problem ($L_i=L(K_i)=s(K_i)L_0$).

\begin{corollary}\label{cor:gradconv}
Under the hypotheses of Corollary~\ref{cor:subsolncomp}, write $u_i=u_{K_i,L_i}$, $u_\lambda=u_{K_\lambda,L_\lambda}$, $R_i=R_{K_i,L_i}$ and $R_\lambda=R_{K_\lambda,L_\lambda}$. Then for every $\xi\in\Snn$ and $\lambda\in[0,1]$:
\begin{enumerate}
\item For all $1<p<\infty$,
\begin{equation}\label{eq:gradconv}
\frac{1}{|\gr u_\lambda(\gr h_{K_\lambda}(\xi))|}
\leq \frac{1-\lambda}{|\gr u_1(\gr h_{K_1}(\xi))|}
+\frac{\lambda}{|\gr u_2(\gr h_{K_2}(\xi))|}.
\end{equation}
\item For $p>n$,
\begin{equation}\label{eq:Rconv_p>n}
R_\lambda(\xi)^{-\frac{1}{p-n}}
\leq (1-\lambda)R_1(\xi)^{-\frac{1}{p-n}}
+\lambda R_2(\xi)^{-\frac{1}{p-n}}.
\end{equation}
In particular, if $p\ge n+1$, then $R_\lambda(\xi)^{-1}\leq (1-\lambda)R_1(\xi)^{-1}+\lambda R_2(\xi)^{-1}$.
\item For $p=n$,
\begin{equation}\label{eq:Rconv_p=n}
R_\lambda(\xi)\geq \min\{R_1(\xi),R_2(\xi)\},
\end{equation}
and when $n=p=2$, $R_\lambda(\xi)\geq \dfrac{(1-\lambda)|\gr u_1|^{-1}R_1+\lambda|\gr u_2|^{-1}R_2}{(1-\lambda)|\gr u_1|^{-1}+\lambda|\gr u_2|^{-1}}$ evaluated at the corresponding boundary points with normal $\xi$.
\end{enumerate}
\end{corollary}
\begin{proof}
By Corollary~\ref{cor:subsolncomp} and \eqref{eq:conprop}, $\{u_\lambda\geq s\}\supseteq (1-\lambda)\{u_1\geq s\}+\lambda\{u_2\geq s\}$ for $0<s<\eps_0$, so taking support functions at $\xi$ gives
\[
h_{u_\lambda}(\xi,s)\geq (1-\lambda)h_{u_1}(\xi,s)+\lambda h_{u_2}(\xi,s).
\]
At $s=0$, equality follows from $h_{u_i}(\xi,0)=h_{K_i}(\xi)$ for $i\in\{1,2,\lambda\}$ and $h_{K_\lambda}=(1-\lambda)h_{K_1}+\lambda h_{K_2}$. Subtracting the $s=0$ identity, dividing by $s>0$, and letting $s\downarrow 0$ yields $\del_t h_{u_\lambda}(\xi,0)\geq (1-\lambda)\del_t h_{u_1}(\xi,0)+\lambda\del_t h_{u_2}(\xi,0)$. Since $\del_t h_u(\xi,0)=-|\gr u(\gr h_K(\xi))|^{-1}$ by Lemma~\ref{lem:ids1}(3) and boundary smoothness, negating both sides proves \eqref{eq:gradconv}.

For (2) and (3), fix $\xi\in\Snn$ and set $a_j=|\gr u_j(\gr h_{K_j}(\xi))|$ and $b_j=\kr_{\del K_j}(\gr h_{K_j}(\xi))^{-1}=\det(\covtwo h_{K_j}(\xi)+h_{K_j}(\xi)\I)$ for $j\in\{1,2,\lambda\}$, so that $R_j=a_j^{p-1}b_j$. By Minkowski's determinant inequality,
\[
b_\lambda^{\frac{1}{n-1}}\geq (1-\lambda)b_1^{\frac{1}{n-1}}+\lambda b_2^{\frac{1}{n-1}}.
\]
If $p>n$, put $\lambda_1=1-\lambda$, $\lambda_2=\lambda$. Applying H\"older's inequality with conjugate exponents $\frac{p-1}{n-1}>1$ and $\frac{p-1}{p-n}>1$ gives
\[
\sum_{i=1}^2\lambda_i a_i^{-1}
=\sum_{i=1}^2\lambda_i \big(b_i^{\frac{1}{n-1}}\big)^{\frac{n-1}{p-1}}\big(R_i^{-\frac{1}{p-n}}\big)^{\frac{p-n}{p-1}}
\leq \left(\sum_{i=1}^2\lambda_i b_i^{\frac{1}{n-1}}\right)^{\frac{n-1}{p-1}}
\left(\sum_{i=1}^2\lambda_i R_i^{-\frac{1}{p-n}}\right)^{\frac{p-n}{p-1}}.
\]
Raising both sides to the power $p-1$ and combining with \eqref{eq:gradconv} and the determinant inequality gives $a_\lambda^{-(p-1)}\leq b_\lambda\left(\sum_{i=1}^2\lambda_i R_i^{-1/(p-n)}\right)^{p-n}$, which is equivalent to \eqref{eq:Rconv_p>n}. When $p\ge n+1$, convexity of $t\mapsto t^{p-n}$ yields $R_\lambda^{-1}\leq (1-\lambda)R_1^{-1}+\lambda R_2^{-1}$.

If $p=n$, then \eqref{eq:gradconv} and the determinant inequality yield
\[
R_\lambda^{\frac{1}{n-1}}=a_\lambda b_\lambda^{\frac{1}{n-1}}
\geq \frac{(1-\lambda)b_1^{\frac{1}{n-1}}+\lambda b_2^{\frac{1}{n-1}}}{(1-\lambda)a_1^{-1}+\lambda a_2^{-1}}
=\frac{(1-\lambda)a_1^{-1}R_1^{\frac{1}{n-1}}+\lambda a_2^{-1}R_2^{\frac{1}{n-1}}}{(1-\lambda)a_1^{-1}+\lambda a_2^{-1}}
\geq \min\big\{R_1^{\frac{1}{n-1}},R_2^{\frac{1}{n-1}}\big\},
\]
proving \eqref{eq:Rconv_p=n} and the $n=p=2$ formula.
\end{proof}
The reciprocal-density convexity for $p\ge n+1$ in part~(2) lies outside the range $1<p<n+1$ of Theorem~\ref{thm:bm}.

\subsection{Proof of the Brunn--Minkowski inequality}
In this subsection we prove Proposition~\ref{prop:condbm} using the limiting characterization in Proposition~\ref{prop:wlimchar} and Corollary~\ref{cor:subsolncomp}, under the geometric and covariance conditions stated below. 

To this end, we require some preparation. 
Given any $K\in\mathcal{N}(L_0)$, let
\begin{equation}\label{eq:omsk}
\Omega_s(K)=\{x\in K:u_K(x)>s\}, 
\end{equation}
for $0\leq s<\eps_0$, using the extension \eqref{eq:extendedu}, where $u_K$ solves the Dirichlet problem \eqref{eq:kdir} with inner body $L_0$.
Thus, $\overline{\Omega_0(K)}=K$ and
$\partial\Omega_s(K)=\{u_K=s\}$ in the ring for $0<s<\eps_0$.
As a consequence of \eqref{eq:wlimchar}, from \eqref{eq:intbd} and \eqref{eq:detkinv}, we have 
\begin{equation}\label{eq:swlimchar}
 \T(K)=\lim_{s\to 0^+}s^{p-1} \int_{\Snn} \frac{h_{\Om_s(K)}(\xi)}{\dist(\gr h_{\Om_s(K)}(\xi),\del K)^{p-1}}\frac{d\xi}{\kr_{\del \Om_s(K)}}.
\end{equation}

For each of the finitely many bodies in a comparison, write $d_{K,s}(\xi)=\dist(\gr h_{\Om_s(K)}(\xi),\del K)$. The uniform version of \eqref{eq:u/d'} and the nonvanishing boundary gradient give
\begin{equation}\label{eq:uniformlevels}
\frac{d_{K,s}(\xi)}{s}\longrightarrow
\frac1{|\gr u_K(\gr h_K(\xi))|},
\qquad
0<c\leq\frac{d_{K,s}(\xi)}s\leq C,
\qquad c\leq\kr_{\del\Om_s(K)}(\xi)\leq C,
\end{equation}
where the constants may depend on the finitely many bodies used but not on $\xi$ or small $s$. All convergence here and in the smooth curvature formula is uniform in $\xi$.

Let us denote the mean and covariance for any $f,g\in L^2(\Snn)$ as follows
\begin{equation}\label{eq:mcov}
\begin{aligned}
\langle f\rangle_{\Snn} &:=\aveint{\Snn}{} f(\xi)\, d\xi;\\
 \Cov_{\Snn} (f,g)&:= \aveint{\Snn}{} \Big(f(\xi)- \langle f\rangle_{\Snn}\Big)
 \Big(g(\xi)- \langle g\rangle_{\Snn}\Big)\, d\xi.
\end{aligned}
\end{equation}
Note that the covariance scales bilinearly and 
$\Cov_{\Snn} (f,c)=0=\Cov_{\Snn} (c,g)$ if $c$ is a constant. 
With respect to these, we have the following elementary identity. 
\begin{lemma}\label{lem:wid}
For any $w_i,f_i\in L^2(\Snn)$, $i\in\{1,2\}$, we have
\begin{equation}\label{eq:covid_general}
\begin{aligned}
\int_{\Snn}
&\big[w_1(\xi)f_1(\xi)+w_2(\xi)f_2(\xi)\big]\,d\xi\\
&={}
\langle w_1\rangle_{\Snn}
\int_{\Snn}f_1(\xi)\,d\xi
+\langle w_2\rangle_{\Snn}
\int_{\Snn}f_2(\xi)\,d\xi\\
&\qquad +\frac{\upomega_n}{2}\Big[
\Cov_{\Snn}\big(w_1-w_2,f_1-f_2\big)
+
\Cov_{\Snn}\big(w_1+w_2,f_1+f_2\big)\Big].
\end{aligned}
\end{equation}
\end{lemma}
\begin{proof}
It follows directly from the definition of covariance that, for any
$f,g\in L^2(\Snn)$,
\begin{equation}\label{eq:cov_product}
\int_{\Snn}f(\xi)g(\xi)\,d\xi
=
\upomega_n
\left[
\langle f\rangle_{\Snn}\langle g\rangle_{\Snn}
+\Cov_{\Snn}(f,g)
\right].
\end{equation}
Indeed, expanding the covariance gives
\[
\Cov_{\Snn}(f,g)
=
\aveint{\Snn}{}f(\xi)g(\xi)\,d\xi
-\langle f\rangle_{\Snn}\langle g\rangle_{\Snn},
\]
which is equivalent to \eqref{eq:cov_product}.
We now use the pointwise algebraic identity
\begin{align*}
w_1f_1+w_2f_2
={}&
\frac12(w_1+w_2)(f_1+f_2)
+\frac12(w_1-w_2)(f_1-f_2).
\end{align*}
Integrating the above and applying \eqref{eq:cov_product} to the two products, we
obtain
\begin{align*}
\int_{\Snn}(w_1f_1+w_2f_2)\,d\xi
={}&
\frac{\upomega_n}{2}\big[
\langle w_1+w_2\rangle_{\Snn}
\langle f_1+f_2\rangle_{\Snn}
+
\Cov_{\Snn}(w_1+w_2,f_1+f_2)\big]\\
&+\frac{\upomega_n}{2}
\big[\langle w_1-w_2\rangle_{\Snn}
\langle f_1-f_2\rangle_{\Snn}
+
\Cov_{\Snn}(w_1-w_2,f_1-f_2)\big].
\end{align*}
The two products of means of the above satisfy
\begin{align*}
&\frac{\upomega_n}{2}\Big[
\langle w_1+w_2\rangle_{\Snn}
\langle f_1+f_2\rangle_{\Snn}+
\langle w_1-w_2\rangle_{\Snn}
\langle f_1-f_2\rangle_{\Snn}\Big]\\
&=
\upomega_n
\Big[
\langle w_1\rangle_{\Snn}\langle f_1\rangle_{\Snn}
+\langle w_2\rangle_{\Snn}\langle f_2\rangle_{\Snn}
\Big]\\
&=
\langle w_1\rangle_{\Snn}\int_{\Snn}f_1\,d\xi
+\langle w_2\rangle_{\Snn}\int_{\Snn}f_2\,d\xi.
\end{align*}
Substituting this equality into the preceding expression yields
\eqref{eq:covid_general} to complete the proof. 
\end{proof}
\begin{corollary}\label{cor:covid0}
For any $w_i,f_i\in L^2(\Snn)$, $i\in\{1,2\}$, if $w_1+w_2=1$ then 
\begin{equation}\label{eq:covid0}
\begin{aligned}
\int_{\Snn}
\big[w_1(\xi)f_1(\xi)+w_2(\xi)f_2(\xi)\big]\,d\xi
&={}
\langle w_1\rangle_{\Snn}
\int_{\Snn}f_1(\xi)\,d\xi
+\langle w_2\rangle_{\Snn}
\int_{\Snn}f_2(\xi)\,d\xi\\
&\quad +\frac{\upomega_n}{2}
\Cov_{\Snn}\big(w_1-w_2,f_1-f_2\big).
\end{aligned}
\end{equation}
\end{corollary}
\begin{proof}
It follows directly from \eqref{eq:covid_general} of Lemma \ref{lem:wid} and $\Cov_{\Snn}\big(1,f_1+f_2\big)=0$. 
\end{proof}

\begin{proposition}\label{prop:condbm}
Let $K_1,K_2\in\mathcal{N}(L_0)$ be given. For any $\lambda\in (0,1)$, let $K_\lambda=(1-\lambda)K_{1} + \lambda K_{2}$  
and for $x\in \del\Om_s(K_\lambda)$, let $x_1\in \del\Om_s(K_1)$ and $x_2\in \del\Om_s(K_2)$ be the points of common normal $\xi\in\Snn$ for small enough $0<s<\eps_0/2$. Write $d_{\lambda,s}(\xi)=\dist(x, \del K_\lambda)$ and $d_{i,s}(\xi)=\dist(x_i, \del K_i)$ for $i\in \{1,2\}$. Suppose that the following conditions hold: 
\begin{enumerate}
\item The curvature-weighted distances satisfy
\begin{equation}\label{eq:curvcond0}
\kr_{\del K_\lambda} d_{\lambda,s}^{p-1}
\leq (1-\lambda)\kr_{\del K_1} d_{1,s}^{p-1}
+ \lambda \kr_{\del K_2} d_{2,s}^{p-1}
+\om(s) s^{p-1},
\end{equation}
where $\omega(s)\geq0$ and $\omega(s)\to0$ as $s\to0^+$. (By Corollary~\ref{cor:gradconv}(2) and \eqref{eq:uniformlevels}, this holds unconditionally when $p\ge n+1$.)
\item The covariance \eqref{eq:mcov} satisfies
\begin{equation}\label{eq:condw}%\tag*{\textup{(W)}}
 \liminf_{s\to0^+}s^{p-1}
 \operatorname{Cov}_{\Snn}\Bigg(
 \frac{(1-\lambda)\kr_{\del K_1}d_{1,s}^{p-1} - \lambda\kr_{\del K_2}d_{2,s}^{p-1}}{(1-\lambda)\kr_{\del K_1}d_{1,s}^{p-1} + \lambda\kr_{\del K_2}d_{2,s}^{p-1}}\,,\,
\left[
 \frac{h_{\Om_s(K_1)}}{\kr_{\del K_1}d_{1,s}^{p-1}}
 -\frac{h_{\Om_s(K_2)}}{\kr_{\del K_2}d_{2,s}^{p-1}}
 \right]\Bigg)\geq0.
\end{equation}
\end{enumerate}
Then, we have $\T(K_\lambda) \geq \min\,\{\T(K_{1}), \T(K_{2})\}$. 
\end{proposition}
\begin{proof}
Recall that we have shown in Corollary \ref{cor:subsolncomp} 
that $u_{K_{\lambda}}\geq u_{\lambda}^{*}$. Recalling \eqref{eq:supcon} and \eqref{eq:conprop}, it 
is equivalent to
\begin{equation}\label{eq:BMp0}
    \{u_{K_{\lambda}}\geq s\} \supseteq \{u_{\lambda}^{*}\geq s\} = (1-\lambda)\{u_{K_{1}}\geq s\} + \lambda \{u_{K_{2}}\geq s\},
\end{equation}
for $0<s<\eps_0$.
In other words, in terms of the superlevel sets \eqref{eq:omsk}, we have 
\begin{equation}\label{eq:BMp1}
\overline{\Omega_{s}(K_{\lambda})} \supseteq (1-\lambda)\overline{\Omega_{s}(K_{1})} + \lambda \overline{\Omega_{s}(K_{2})}.
\end{equation}
Therefore, using \eqref{eq:suppminkadd} and the above, we have 
\begin{equation}\label{eq:inq1}
\begin{aligned}
h_{\Omega_{s}(K_{\lambda})}(\xi)
        \geq (1-\lambda)h_{\Omega_{s}(K_{1})}(\xi) + \lambda h_{\Omega_{s}(K_{2})}(\xi).
    \end{aligned}
\end{equation}

Write $h_{i,s}=h_{\Om_s(K_i)}$ and $\mathcal K_{i,s}(\xi)=\mathcal K_{\partial\Omega_s(K_i)}(\nabla h_{i,s}(\xi))$ for $i\in\{1,2,\lambda\}$, pulling all curvatures back by their respective inverse Gauss maps. By smooth level convergence and \eqref{eq:uniformlevels}, define the uniformly bounded quantities
\[
F_{i,s}:=\frac{s^{p-1}h_{i,s}}{\kr_{i,s}d_{i,s}^{p-1}},
\qquad
w_{1,s}:=\frac{(1-\lambda)\kr_{1,s}d_{1,s}^{p-1}}
{(1-\lambda)\kr_{1,s}d_{1,s}^{p-1}+\lambda\kr_{2,s}d_{2,s}^{p-1}},
\quad w_{2,s}=1-w_{1,s}.
\]
The curvature condition \eqref{eq:curvcond0}, divided by $s^{p-1}$, together with \eqref{eq:inq1} and the positive lower bounds in \eqref{eq:uniformlevels}, gives
\begin{align*}
F_{\lambda,s}
&=\frac{h_{\lambda,s}}{\kr_{\lambda,s}(d_{\lambda,s}/s)^{p-1}}\\
&\geq \frac{(1-\lambda)h_{1,s}+\lambda h_{2,s}}{(1-\lambda)\kr_{1,s}(d_{1,s}/s)^{p-1}+\lambda\kr_{2,s}(d_{2,s}/s)^{p-1}}+o(1)\\
&=w_{1,s}F_{1,s}+w_{2,s}F_{2,s}+o(1)
\end{align*}
uniformly on $\Snn$.

Apply Corollary~\ref{cor:covid0} before taking a combined lower limit. Replacing $\kr_{\del K_i}$ by $\kr_{i,s}$ in the normalized covariance of \eqref{eq:condw} changes it by $o(1)$, again by \eqref{eq:uniformlevels}. Therefore
\begin{align*}
\T(K_\lambda)
&\geq\liminf_{s\downarrow0}\left[
\langle w_{1,s}\rangle_{\Snn}\int_{\Snn}F_{1,s}\,d\xi
+\langle w_{2,s}\rangle_{\Snn}\int_{\Snn}F_{2,s}\,d\xi\right.\\
&\hspace{35mm}\left.
+\frac{\upomega_n}{2}\Cov_{\Snn}(w_{1,s}-w_{2,s},F_{1,s}-F_{2,s})\right]\\
&\geq\min\{\T(K_1),\T(K_2)\}.
\end{align*}

\end{proof}

We next give a boundary formulation of the covariance identity. In the next proposition and remark, $T$ denotes either $\T(\cdot;L_0)$ or $\T(K)=\T(K;L(K))$, and $R_K$ denotes the corresponding density. 

\begin{proposition}\label{prop:boundaryidentity} 
Let $K_\lambda=\alpha K_1+\beta K_2$, with $\alpha=1-\lambda$, $\beta=\lambda$, and let $R_1,R_2,R_\lambda$ be the positive continuous densities for either of these two constructions, used consistently. Write $T(K)=\int h_KR_K\,d\xi$ for the corresponding functional. Set
\[
D=\alpha R_2+\beta R_1,\qquad H=\frac{R_1R_2}{D},\qquad
w_1=\frac{\alpha R_2}{D},\quad w_2=\frac{\beta R_1}{D},
\]
\[
E=\int_{\Snn}h_{K_\lambda}(R_\lambda-H)\,d\xi.
\]
Then, with the normalized spherical covariance of \eqref{eq:mcov},
\begin{equation}\label{eq:boundaryidentity}
\begin{aligned}
T(K_\lambda)&=\langle w_1\rangle_{\Snn}T(K_1)+\langle w_2\rangle_{\Snn}T(K_2)\\
&\quad+\frac{\upomega_n}{2}\Cov_{\Snn}(w_1-w_2,h_{K_1}R_1-h_{K_2}R_2)+E.
\end{aligned}
\end{equation}
\end{proposition}
\begin{proof}
The support identity $h_{K_\lambda}=\alpha h_{K_1}+\beta h_{K_2}$ gives
$h_{K_\lambda}H=w_1h_{K_1}R_1+w_2h_{K_2}R_2$.
Integrate, add $E$, and use Corollary~\ref{cor:covid0}, and the fact that $w_1+w_2=1$.
\end{proof}

\begin{remark}\label{thm:bm1}
With the notation of Proposition~\ref{prop:boundaryidentity}, the condition
\begin{equation}\label{eq:condwcbdry}
\Cov_{\Snn}\left(\frac{\alpha R_2-\beta R_1}{\alpha R_2+\beta R_1},
 h_{K_1}R_1-h_{K_2}R_2\right)\geq-\frac{2E}{\upomega_n}
\end{equation}
implies
\begin{equation}\label{eq:bm1}
T((1-\lambda)K_1+\lambda K_2)\geq\min\{T(K_1),T(K_2)\}
\end{equation}
by \eqref{eq:boundaryidentity}, since its first two terms form a convex combination of $T(K_1)$ and $T(K_2)$.
\end{remark}
If $T(K_1)=T(K_2)$, condition \eqref{eq:condwcbdry} is exactly equivalent to $T(K_\lambda)\geq T(K_1)$. To obtain the power inequality, we first compare bodies with the same functional value and then use homogeneity for their dilates.

\begin{definition}\label{def:normalizedclass}
Let $1<p<n+1$ and $q=n-p+1>0$. In this definition and the proof of Theorem~\ref{thm:bm}, we return to the problem of Definition~\ref{def:hatselection}, with inner body $L(K)$. For $K\in\mathcal U$, we write $u_K=u_{K,L(K)}$, $\mu_K=\mu_{K,L(K)}$, $R_K=R_{K,L(K)}$ and $\T(K)=\T(K;L(K))$. We call $K\in\mathcal U$ \emph{normalized} if $\T(K)=\T(K_0)$.

A class $\mathcal D\subset\mathcal U$ containing $K_0$ is called \emph{admissible} if it has the following properties. It is closed under positive dilations and is Minkowski convex, i.e. for any $K_1,K_2\in\mathcal D$,
\[
aK_1\in\mathcal D\quad\text{for all }a>0,
\qquad
[K_1,K_2]=\{(1-\lambda)K_1+\lambda K_2:\lambda\in[0,1]\}\subseteq\mathcal D.
\]
For any normalized $K_1,K_2\in\mathcal D$ and $0<\lambda<1$, with $K_\lambda=(1-\lambda)K_1+\lambda K_2$ and $H=\dfrac{R_{K_1}R_{K_2}}{(1-\lambda)R_{K_2}+\lambda R_{K_1}}$, the following conditions hold:
\begin{enumerate}
\item (Integral harmonic-mean condition)
\begin{equation}\label{eq:curvbound}
E:=\int_{\Snn}h_{K_\lambda}\big(R_{K_\lambda}-H\big)\,d\xi\geq 0.
\end{equation}
(In particular, \eqref{eq:curvbound} follows from the pointwise reciprocal-density convexity $R_{K_\lambda}^{-1}\leq (1-\lambda)R_{K_1}^{-1}+\lambda R_{K_2}^{-1}$.)
\item (Covariance condition) the covariance \eqref{eq:mcov} satisfies
\begin{equation}\label{eq:normalizedcovariance}
\Cov_{\Snn}\left(
\frac{(1-\lambda)R_{K_2}-\lambda R_{K_1}}
{(1-\lambda)R_{K_2}+\lambda R_{K_1}},
h_{K_1}R_{K_1}-h_{K_2}R_{K_2}\right)\geq0.
\end{equation}
\end{enumerate}
We say that $\mathcal D$ satisfies the \emph{strictness condition} if, for any normalized $K_1\ne K_2$ in $\mathcal D$ and $0<\lambda<1$, either \eqref{eq:curvbound} is strict or \eqref{eq:normalizedcovariance} is strict.
\end{definition}

\begin{proof}[Proof of Theorem \ref{thm:bm}]
The endpoint cases $\lambda=0,1$ are immediate. Fix normalized $K_1',K_2'\in\mathcal D$ and $0<\theta<1$, and put $K_\theta'=(1-\theta)K_1'+\theta K_2'$. Apply \eqref{eq:boundaryidentity} with $T=\T$, the densities $R_{K_1'},R_{K_2'},R_{K_\theta'}$ for the inner bodies $L(K_1'),L(K_2'),L(K_\theta')$, respectively, and mixing parameter $\theta$. Here $E$ and $H$ refer to this application. Conditions \eqref{eq:curvbound} and \eqref{eq:normalizedcovariance} give $E\geq0$ and nonnegative covariance. Since $\T(K_1')=\T(K_2')=\T(K_0)$, we obtain
\[
\T(K_\theta')\geq\T(K_0).
\]
For $K_1,K_2\in\mathcal D$ and $0<\lambda<1$, write $K_1=aK_1'$ and $K_2=bK_2'$ with $K_1',K_2'$ normalized.
Put $S=(1-\lambda)a+\lambda b$ and $\theta=\lambda b/S$. Then $K_\lambda=SK'_\theta$.
Homogeneity, Lemma~\ref{prop:hathom}, gives
\[
\T(K_\lambda)=S^q\T(K_\theta')\geq\T(K_0)S^q.
\]
Taking the $q$-th root proves \eqref{eq:bm0}, since
$\T(K_1)^{1/q}=a\T(K_0)^{1/q}$ and $\T(K_2)^{1/q}=b\T(K_0)^{1/q}$.

Under the strictness condition, if $K_1'\ne K_2'$, either the covariance term is positive or $E>0$. Thus the normalized inequality, and hence \eqref{eq:bm0}, is strict. If $K_1'=K_2'$, equality follows from homogeneity. Since every body has a unique normalized dilate, equality is equivalent to dilation about the chosen origin.
\end{proof}

We next compute the second variation for $p=2$ and $n\ge 3$ near concentric balls.

\begin{proposition}\label{prop:p2threshold}
Let $p=2$, $n\ge 3$, $K_0=B_1(0)$, and $L_0=B_r(0)$ with $0<r<1$, and consider the homogeneous functional $\T(K)=\T(K;L(K))$ of Definition~\ref{def:hatselection} with inner value $\eps_0>0$. For a nonzero real-valued spherical harmonic $Y_\ell$ on $\Snn$ of degree $\ell\ge 1$, let $K_\eps$ be the convex body with support function $h_{K_\eps}(\xi)=1+\eps Y_\ell(\xi)$ for $|\eps|$ small. Then:
\begin{enumerate}
\item The second variation of $\T(K_\eps)$ at $\eps=0$ is given by
\begin{equation}\label{eq:secondvar}
\left.\frac{d^2}{d\eps^2}\right|_{\eps=0}\T(K_\eps)
=\eps_0\upomega_n A(r)\Big[A(r)\big(2c_\ell(r)+n-1+\ell(\ell+n-2)\big)
-2\big(c_\ell(r)+\ell(\ell+n-2)\big)\Big]\aveint{\Snn}{}Y_\ell^2\,d\xi,
\end{equation}
where
\[
A(r)=\frac{n-2}{r^{-(n-2)}-1},
\qquad
c_\ell(r)=\frac{\ell+(\ell+n-2)r^{2\ell+n-2}}{1-r^{2\ell+n-2}}.
\]
\item Set $r_*=(n-1)^{-\frac{1}{n-2}}$. If $0<r<r_*$, then the bracket in \eqref{eq:secondvar} is strictly negative for every $\ell\ge 1$. At $r=r_*$ it vanishes for $\ell=1$ and is strictly negative for $\ell\ge2$. For $r>r_*$ it is strictly positive for $\ell=1$.
\item For $\ell\ge 1$, $K_1=K_\eps$ and $K_2=B_1$, and sufficiently small $0<|\eps|$, the inequality in \eqref{eq:Thadamard} for the ordering $(K_\eps,B_1)$ is strict, that is,
\[
\left.\frac{d}{d\nu}\right|_{0+}\T((1-\nu)K_1+\nu K_2)
<q\int_{\Snn}(h_{K_2}-h_{K_1})\,d\mu_{K_1}.
\]
\end{enumerate}
\end{proposition}
\begin{proof}
Since $\T$ scales linearly in $\eps_0$ for $p=2$, we may take $\eps_0=1$ and multiply by $\eps_0$ at the end. Put $\varphi=Y_\ell$. Because $\int_{\Snn}\varphi\,d\xi=0$, we have $s(K_\eps)=1$ and $L(K_\eps)=B_r$ identically for all small $\eps$. Writing $x(\xi)=\gr h_{K_\eps}(\xi)=(1+\eps\varphi)\xi+\eps\cov\varphi$, the radial function $\rho(\theta)$ of $\del K_\eps$ satisfies
\[
\rho(\theta)=1+\eps\varphi(\theta)-\frac{\eps^2}{2}|\cov\varphi(\theta)|^2+O(\eps^3).
\]
Pulling back to the fixed annulus by a smooth family of diffeomorphisms that is the identity near the inner boundary gives a uniformly elliptic Dirichlet problem with coefficients depending smoothly on $\eps$. Invertibility of this problem and Schauder estimates give smooth dependence of the solutions and their boundary gradients on $\eps$ in the pulled-back coordinates, which justifies the expansions below.

Write $\Phi_\eps$ for these diffeomorphisms, with $\Phi_0$ the identity, and put $V=\left.\partial_\eps\Phi_\eps\right|_{\eps=0}$. The derivative at fixed spatial points is
\[
u_1=\left.\partial_\eps(u_\eps\circ\Phi_\eps)\right|_{\eps=0}-\gr u_0\cdot V,
\qquad u_\eps=u_{K_\eps}.
\]
It is harmonic in the reference annulus. The radial solution $u_0(\rho)=\frac{\rho^{-(n-2)}-1}{r^{-(n-2)}-1}$ satisfies $u_0(1)=0$, $u_0'(1)=-A(r)$, and $u_0''(1)=(n-1)A(r)$. Imposing $u_\eps(\rho(\theta)\theta)=0$ at order $\eps$ gives $u_1(\theta)=A(r)\varphi(\theta)$ on $\del B_1$ and $u_1=0$ on $\del B_r$, whence
\[
u_1(\rho\theta)=A(r)\frac{\rho^\ell-r^{2\ell+n-2}\rho^{-(\ell+n-2)}}{1-r^{2\ell+n-2}}\varphi(\theta),
\qquad
\del_\rho u_1(1,\theta)=A(r)c_\ell(r)\varphi(\theta).
\]
We compute the two terms in
\[
\T(K_\eps)=\int_{\del K_\eps}|\gr u_\eps|\,d\h^{n-1}
+\eps\int_{\Snn}\varphi(\xi)R_{K_\eps}(\xi)\,d\xi
\]
to order $\eps^2$. The first integral equals $\operatorname{Cap}_2(B_r,K_\eps)$. Its Hadamard first variation is
\[
-\int_{\del K_\eps}|\gr u_\eps|^2\varphi(\g_{K_\eps})\,d\h^{n-1}.
\]
Using
\begin{align*}
|\gr u_\eps(\gr h_{K_\eps}(\xi))|&=A(r)[1-\eps(n-1+c_\ell(r))\varphi(\xi)]+O(\eps^2),\\
\det(\covtwo h_{K_\eps}+h_{K_\eps}\I)&=1+\eps(n-1-\ell(\ell+n-2))\varphi(\xi)+O(\eps^2),
\end{align*}
we obtain
\[
R_{K_\eps}(\xi)=A(r)\Big[1-\eps\big(c_\ell(r)+\ell(\ell+n-2)\big)\varphi(\xi)\Big]+O(\eps^2),
\]
and differentiating $-\int_{\Snn}\varphi|\gr u_\eps|^2\det(\covtwo h_{K_\eps}+h_{K_\eps}\I)\,d\xi$ at $\eps=0$ gives
\[
\left.\frac{d^2}{d\eps^2}\right|_{\eps=0}\int_{\del K_\eps}|\gr u_\eps|\,d\h^{n-1}
=\upomega_n A(r)^2\big(2c_\ell(r)+n-1+\ell(\ell+n-2)\big)\aveint{\Snn}{}\varphi^2\,d\xi.
\]
Meanwhile, $\eps\int_{\Snn}\varphi R_{K_\eps}\,d\xi=-\eps^2\upomega_n A(r)(c_\ell(r)+\ell(\ell+n-2))\aveint{\Snn}{}\varphi^2+O(\eps^3)$. Adding the two contributions proves \eqref{eq:secondvar}.

Notice that $A(r)=1$ is equivalent to $r=r_*$. For $0<r<r_*$, $A(r)<1$ and $\ell(\ell+n-2)\ge n-1$ for every $\ell\ge 1$, so the bracket is bounded above by $2(A(r)-1)(c_\ell(r)+\ell(\ell+n-2))<0$. At $r=r_*$ the bracket equals $n-1-\ell(\ell+n-2)$. For $\ell=1$ it equals $2(A(r)-1)(c_1(r)+n-1)$, which is positive when $r>r_*$. This proves (2). For $r>r_*$ the first variation vanishes at the ball in degree one, so the positive second variation gives $\T(K_{\pm\eps})>\T(B_1)$ for sufficiently small $0<|\eps|$. Thus \eqref{eq:bm0} fails on $\mathcal U$ at the midpoint $B_1$ of these two bodies.

Finally, for $K_1=K_\eps$ ($\ell\ge 1$) and $K_2=B_1$, we have $(1-\nu)K_1+\nu K_2=K_{(1-\nu)\eps}$. Thus
\begin{align*}
\left.\frac{d}{d\nu}\right|_{0+}\T(K_{(1-\nu)\eps})
&=-\eps\left.\frac{d}{dt}\right|_{t=\eps}\T(K_t),\\
q\int_{\Snn}(h_{K_2}-h_{K_1})\,d\mu_{K_1}
&=(n-1)\int_{\Snn}(-\eps\varphi)R_{K_\eps}\,d\xi.
\end{align*}
Put $\gamma_\ell=c_\ell(r)+\ell(\ell+n-2)$ and $M_\ell=2c_\ell(r)+n-1+\ell(\ell+n-2)$. Using \eqref{eq:secondvar} and the expansion of $R_{K_\eps}$, and restoring the factor $\eps_0$, gives
\begin{align*}
&\left.\frac{d}{d\nu}\right|_{0+}\T((1-\nu)K_1+\nu K_2)
-q\int_{\Snn}(h_{K_2}-h_{K_1})\,d\mu_{K_1}\\
&\qquad=-\eps_0 A(r)\big[(n-3)\gamma_\ell+A(r)M_\ell\big]\eps^2
\int_{\Snn}Y_\ell^2\,d\xi+O(\eps^3).
\end{align*}
The coefficient of $\eps^2$ is strictly negative for $n\ge3$ and $0<r<1$. This proves (3) for sufficiently small $0<|\eps|$. For the reverse ordering $(B_1,K_\eps)$ both sides of \eqref{eq:Thadamard} vanish, since the first variation in direction $Y_\ell$ is zero at the ball and $R_{B_1}$ is constant. For $(K_\eps,K_{2\eps})$ the direction at $K_\eps$ is the negative of the direction towards $B_1$. Both sides change sign by differentiability along the path $t\mapsto K_t$, so the upper bound fails for this pair.
\end{proof}

\noindent\emph{Remark.} For the concentric data of Proposition~\ref{prop:p2threshold}, the only admissible class is $\{aB_1:a>0\}$. To see this, let $K$ belong to an admissible class and rescale it so that $s(K)=1$. Set $f=h_K-1$, which has spherical mean zero, and let $K_t=(1-t)B_1+tK$. Then $h_{K_t}=1+tf$ and $L(K_t)=B_r$. Write $f=\sum_{\ell\ge1}f_\ell$ for its spherical harmonic expansion and put
\[
R_0=R_{B_1}=\eps_0 A(r),\qquad Gf=\sum_{\ell\ge1}\gamma_\ell f_\ell,\qquad
\gamma_\ell=c_\ell(r)+\ell(\ell+n-2).
\]
By linearity of the harmonic Dirichlet problem for the first variation, the density calculation in Proposition~\ref{prop:p2threshold} gives $R_{K_t}=R_0(1-tGf)+O(t^2)$. The linearized formula applies term by term since $f$ is smooth and $\gamma_\ell=O(\ell^2)$. Smooth dependence on $t$ gives the uniform remainder for this fixed $f$. Both $f$ and $Gf$ have mean zero, so $\T(K_t)=\T(B_1)+O(t^2)$. Hence the normalized body $d_tK_t$, where $d_t=(\T(B_1)/\T(K_t))^{1/q}=1+O(t^2)$ and $q=n-1$, has the same first-order expansions of its support function and density by Lemma~\ref{prop:hathom}.

In \eqref{eq:normalizedcovariance} take $K_1=d_tK_t$, $K_2=B_1$ and $\lambda=1/2$. The two arguments of the covariance are $tGf/2+O(t^2)$ and $tR_0(f-Gf)+O(t^2)$, respectively. Orthogonality of the spherical harmonics gives
\[
-\frac{R_0t^2}{2}\sum_{\ell\ge1}\gamma_\ell(\gamma_\ell-1)
\langle f_\ell^2\rangle_{\Snn}+O(t^3)
\]
for the covariance. Since $\gamma_\ell>1$ for every $\ell\ge1$, the covariance is negative for all sufficiently small $t>0$ whenever $f\ne0$. Admissibility puts both $K_t$ and $d_tK_t$ in the class, so $f=0$. Conversely, the centered dilates form an admissible class since their only normalized member is $B_1$.

\subsection{The case for equality with a fixed inner body}

In this subsection the inner body is fixed at $L_0$. We write $u_K=u_{K,L_0}$, $\mu_K=\mu_{K,L_0}$, $R_K=R_{K,L_0}$ and $\T(K)=\T(K;L_0)$. For concentric balls $B_r\Subset B_R$, $\T(B_R;B_r)\to+\infty$ as $R\downarrow r$, so inclusion monotonicity in the increasing direction fails in general. Proposition~\ref{prop:eqcase} assumes strict comparison only for the level bodies $M(t)$ and $C(t)$ defined there. We write the power inequality for a pair $A,B$ in the form
\begin{equation}\label{eq:fixedcorebm}
\T((1-\lambda)A+\lambda B)^{1/q}
\geq(1-\lambda)\T(A)^{1/q}+\lambda\T(B)^{1/q},\qquad q=n-p+1>0.
\end{equation}
\begin{proposition}\label{prop:eqcase}
Let $1<p<n+1$, $K_1,K_2\in\mathcal N(L_0)$, $0<\lambda<1$ and $K_\lambda=(1-\lambda)K_1+\lambda K_2$. Suppose equality holds in \eqref{eq:fixedcorebm} for $(A,B)=(K_1,K_2)$ at this $\lambda$. For some $0<\delta<\eps_0$ and each $0<t<\delta$, put
\[
A_i(t)=\overline{\Om_t(K_i)},\qquad C(t)=\overline{\Om_t(K_\lambda)},\qquad
M(t)=(1-\lambda)A_1(t)+\lambda A_2(t).
\]
For every such $t$, assume that \eqref{eq:fixedcorebm} holds for $(A,B)=(A_1(t),A_2(t))$ at the same $\lambda$, that $\T(M(t);L_0)<\T(C(t);L_0)$ whenever $M(t)\subsetneq C(t)$, and that
\begin{equation}\label{eq:geomcond}
\frac{\T(\Om_t(K_\lambda))}{\T(K_\lambda)}
\leq \min\left\{\frac{\T(\Om_t(K_1))}{\T(K_1)},\,\frac{\T(\Om_t(K_2))}{\T(K_2)}\right\}.
\end{equation}
Then, for any $0\leq t<\delta$, we have
	\begin{equation}\label{eq:eqomlin}
		\Om_t(K_\lambda)=(1-\lambda)\Om_t(K_1)+\lambda\Om_t(K_2),
	\end{equation}
and, for the supremal convolution of \eqref{eq:supcon},
\[
\min\{u_\lambda^*,\delta\}=\min\{u_{K_\lambda},\delta\}.
\]
Thus the functions agree in the region adjacent to the outer boundary below level $\delta$.
\end{proposition}

In particular, \eqref{eq:geomcond} is implied by the pointwise bound
\[
\sup_{\xi\in\Snn}\frac{h_{\Om_t(K_\lambda)}(\xi)R_{\Om_t(K_\lambda)}(\xi)}{h_{K_\lambda}(\xi)R_{K_\lambda}(\xi)}
\leq \min_{i\in\{1,2\}}\inf_{\xi\in\Snn}\frac{h_{\Om_t(K_i)}(\xi)R_{\Om_t(K_i)}(\xi)}{h_{K_i}(\xi)R_{K_i}(\xi)}.
\]

\begin{proof}
	If we denote $\Om_t^*=\{u_\lambda^*>t\}$, then recalling
	\eqref{eq:conprop} and Corollary \ref{cor:subsolncomp}, we have
	\[
	\Om_t(K_\lambda)\supseteq\Om_t^*
	=(1-\lambda)\Om_t(K_1)+\lambda\Om_t(K_2).
	\]
Using this inclusion, the comparison assumption for $M(t)\subseteq C(t)$, and \eqref{eq:fixedcorebm} for $A_1(t),A_2(t)$, we have
	\begin{equation}\label{eq:eqcase1}
		\begin{aligned}
			\T(\Om_t(K_\lambda))^{1/q}
			&\geq\T(\Om_t^*)^{1/q}                                                \\
			&=\T\big((1-\lambda)\Om_t(K_1)+\lambda\Om_t(K_2)\big)^{1/q}          \\
			&\geq(1-\lambda)\T(\Om_t(K_1))^{1/q}
			+\lambda\T(\Om_t(K_2))^{1/q},
		\end{aligned}
	\end{equation}
	where $q=n-p+1>0$. Set $\Lambda_0(t):=\min_{i\in\{1,2\}}\big(\T(\Om_t(K_i))/\T(K_i)\big)^{1/q}$. Using the equality hypothesis in \eqref{eq:fixedcorebm} and \eqref{eq:geomcond}, we obtain
	\begin{equation}\label{eq:eqcase2}
		\begin{aligned}
			(1-\lambda)\T(\Om_t(K_1))^{1/q}+\lambda\T(\Om_t(K_2))^{1/q}
			&\geq \Lambda_0(t)\Big[(1-\lambda)\T(K_1)^{1/q}+\lambda\T(K_2)^{1/q}\Big]\\
			&=\Lambda_0(t)\T(K_\lambda)^{1/q}
			\geq \T(\Om_t(K_\lambda))^{1/q}.
		\end{aligned}
	\end{equation}
	Combining \eqref{eq:eqcase1} and \eqref{eq:eqcase2}, we conclude that
	equality holds throughout. By the strict comparison assumption, equality
	in the first inequality of \eqref{eq:eqcase1} gives \eqref{eq:eqomlin}.
	Since this holds for every $0<t<\delta$, the corresponding super-level sets agree and
	hence $\min\{u_\lambda^*,\delta\}=\min\{u_{K_\lambda},\delta\}$. The proof is complete.
\end{proof}

We now prove that equality of the level-set support functions forces the outer bodies to coincide when the inner body is fixed. 

\begin{proposition}\label{prop:conditionalhomothety}
Let $1<p<\infty$, $K_1,K_2\in\mathcal N(L_0)$, $0<\lambda<1$, and $K_\lambda=(1-\lambda)K_1+\lambda K_2$. Suppose that, for some $0<\delta<\eps_0$,
\[
h_{u_{K_\lambda}}(\xi,t)
=(1-\lambda)h_{u_{K_1}}(\xi,t)+\lambda h_{u_{K_2}}(\xi,t),
\qquad (\xi,t)\in\Snn\times(0,\delta).
\]
Then $K_1=K_2$. In particular, this conclusion holds under the hypotheses of Proposition~\ref{prop:eqcase}.
\end{proposition}
\begin{proof}
For any 
$h\in C^2(\Snn\times(0,\delta))$ satisfying $\partial_t h<0$ and $\covtwo h+h\I>0$, let us denote 
$$ \LL [h](\xi,t) = \frac{1}{(-\del_{t}h)^{p-1}}\left[\Tr\Big[\inv{(\covtwo h + h\I)}\Big] +\frac{(p-1)}{(\del_{t}h)^{2}}\left\{\inp{\inv{(\covtwo h + h\I)}\cov\del_{t}h}{\cov\del_{t}h}-\del_{t}^2h\right\}\right],$$
so that for any $u\in C^2(\Om\cap N)$ that solves \eqref{eq:omdir1}, we have 
$\LL[h_u](\xi, t)=-\lap_p u \big(\inv{\Theta_u}(\xi, t)\big)$ for all $0<t<\delta$
from Proposition \ref{prop:plaphu}, where $h_u$ is as in \eqref{eq:hu} and 
$\Theta_u$ is as in \eqref{eq:bij}. Let us denote 
$h_i(\cdot, t)=h_{u_{K_i}}(\cdot, t)$ for $i\in\{1,2\}$ and 
$$h^*_\lambda (\cdot, t)=(1-\lambda)h_1(\cdot, t)+\lambda h_2(\cdot, t),$$ similarly as in the proof of Proposition \ref{prop:subsoln}. Also, let us denote $h_\lambda (\cdot, t) = h_{u_{K_\lambda}}(\cdot, t)$. From $p$-harmonicity $\LL [h_1]=\LL[h_2]=0=\LL[h_\lambda]$.
By hypothesis, $h^*_\lambda (\cdot, t)=h_\lambda (\cdot, t)$. Therefore, we have 
\begin{equation}\label{eq:lhid}
\LL[(1-\lambda)h_1+\lambda h_2](\cdot, t)=0= (1-\lambda)\LL[h_1](\cdot, t)+\lambda \LL[h_2](\cdot, t),
\end{equation}
for all $0<t<\delta$. As $\del^2_t h_\lambda =(1-\lambda)\del^2_t h_1+\lambda \del^2_t h_2$, therefore, in terms of $M_i, t_i, z_i$ and $M_\lambda, t_\lambda, z_\lambda$ as denoted in Proposition \ref{prop:subsoln}, it is easily seen that \eqref{eq:lhid} yields 
\begin{align*}
t_{\lambda}^{2}\,\mathrm{Tr}(M^{-1}_{\lambda}) +(p-1) \langle M_{\lambda}^{-1}z_{\lambda}, z_{\lambda}\rangle 
&= (1-\lambda)\left[ t_{1}^{2}\,\mathrm{Tr}(M_{1}^{-1})+(p-1) \inp{M^{-1}_{1}z_{1}}{z_{1}} \right]\\
&\quad + \lambda \left[ t_{2}^{2}\,\mathrm{Tr}(M_{2}^{-1})+(p-1) \inp{M^{-1}_{2}z_{2}}{z_{2}} \right].
\end{align*}
Thus, it corresponds to the equality case of Lemma \ref{lem:cstech}, leading to $M_1/t_1=M_2/t_2$ and $M_1^{-1}z_1=M_2^{-1}z_2$. Together these imply $z_1/t_1=z_2/t_2$. Recalling the notations of $z_i$ and  $t_i$ from the proof of Proposition \ref{prop:subsoln}, 
we have 
\[
\frac{\cov\del_t h_1(\xi,t)}{\del_t h_1(\xi,t)}
=\frac{\cov\del_t h_2(\xi,t)}{\del_t h_2(\xi,t)}.
\]
Equivalently,
$\cov\log(\del_t h_1(\xi,t)/\del_t h_2(\xi,t))=0$.
Therefore, $\del_t h_1(\xi,t)=C(t)\del_t h_2(\xi,t)$ for some constant
$C(t)>0$, and hence $M_1=C(t)M_2$. Recalling the notation of $M_i$ from
the proof of Proposition \ref{prop:subsoln} and \eqref{eq:grhcov}, we have
\[
\inv{\mathcal W}_{\Om_t(K_1)}(\xi)
=C(t)\inv{\mathcal W}_{\Om_t(K_2)}(\xi),
\qquad \xi\in\Snn.
\]
Thus, for $0<t<\delta$, the $1$-homogeneous function $f=h_1(\cdot,t)-C(t)h_2(\cdot,t)$ satisfies $\covtwo f+f\I=0$ on $\Snn$. By (2) of Lemma~\ref{lem:sphder}, $D^2f(\xi)=0$ for every $\xi\in\Snn$. Since $D^2f$ is homogeneous of degree $-1$ and $\R^n\mns\{0\}$ is connected, $\gr f$ is a constant vector $v(t)$, and Euler's identity gives $f(y)=\inp{\gr f(y)}{y}=\inp{v(t)}{y}$. Therefore
\[
h_1(\xi,t)=C(t)h_2(\xi,t)+\inp{v(t)}{\xi},
\qquad
\overline{\Om_t(K_1)}=C(t)\overline{\Om_t(K_2)}+v(t).
\]
Since $h_1,h_2\in C^2(\Snn\times(0,\delta))$ and $\Om_t(K_2)$ has nonempty interior, evaluating this relation at $n+1$ directions $\xi\in\Snn$ for which $h_2(\cdot,t),\xi_1,\ldots,\xi_n$ are linearly independent shows that $C(t)$ and $v(t)$ are $C^1$ in $t$. Differentiating $h_1(\xi,t)=C(t)h_2(\xi,t)+\inp{v(t)}{\xi}$ with respect to $t$ and subtracting $\del_t h_1(\xi,t)=C(t)\del_t h_2(\xi,t)$ yields
\[
C'(t)h_{\Om_t(K_2)}(\xi)+\inp{v'(t)}{\xi}=0,
\qquad \forall\,\xi\in\Snn.
\]
Because $\Om_t(K_2)$ has nonempty interior, $h_{\Om_t(K_2)}$ is not linear on $\Snn$; hence $C'(t)\equiv 0$ and $v'(t)\equiv 0$. Thus $C(t)\equiv C>0$ and $v(t)\equiv v\in\R^n$ are constant on $(0,\delta)$.

Passing to $t\to 0^+$ gives $K_1=CK_2+v$, and on the open collar $\{0<u_{K_2}<\delta\}$ we have $u_{K_1}(Cx+v)=u_{K_2}(x)$. Since $L_0\Subset\mathring K_1$, we have $C^{-1}(L_0-v)\Subset\mathring K_2$. Both $w(x):=u_{K_1}(Cx+v)$ and $u_{K_2}(x)$ solve the $p$-Laplace equation with nonvanishing gradient. The equation is analytic and uniformly elliptic away from zero gradient, so both functions are real-analytic in their respective open rings by~\cite{Morrey58}. Therefore $w=u_{K_2}$ on the connected component $U_0$ of $\mathring K_2\setminus(L_0\cup C^{-1}(L_0-v))$ adjacent to $\del K_2$. If the inner bodies $A=L_0$ and $B=C^{-1}(L_0-v)$ differ, interchange them if necessary and choose $\xi$ with $h_A(\xi)>h_B(\xi)$. A supporting point $x_0\in\del A$ in direction $\xi$ lies outside $B$. The ray $x_0+s\xi$, for $s>0$ until it meets $\del K_2$, avoids both inner bodies, so $x_0\in\del U_0\cap\mathring K_2$. At this point one of $w,u_{K_2}$ equals $\eps_0$ while the other is strictly less than $\eps_0$ by the strong maximum principle, a contradiction. Hence $CL_0+v=L_0$. Comparing volumes and centroids forces $C=1$ and $v=0$, and therefore $K_1=K_2$.
The last assertion follows from Proposition~\ref{prop:eqcase}.
\end{proof}
\begin{remark}
If the same level-set equality $h_\lambda^*(\cdot,t)=h_\lambda(\cdot,t)$ on $(0,\delta)$ holds for the homogeneous selection $L(K)=s(K)L_0$ of Definition~\ref{def:hatselection}, then the identical proof yields $L(K_1)=CL(K_2)+v$, i.e. $s(K_1)L_0=Cs(K_2)L_0+v$. Comparing volumes and centroids gives $s(K_1)=Cs(K_2)$ and $v=0$, so $K_1=CK_2$ (dilation about the origin).
\end{remark}

Now we prove the uniqueness statement for the Minkowski problem for $p$-harmonic measures under the hypotheses of Theorem~\ref{thm:pharminkunique}. We return to the problem of Definition~\ref{def:hatselection}, with inner body $L(K)$, $u_K=u_{K,L(K)}$, $\mu_K=\mu_{K,L(K)}$ and $\T(K)=\T(K;L(K))$.
\begin{proof}[Proof of Theorem \ref{thm:pharminkunique}]
Suppose that $\bar\Om_1,\bar\Om_2\in\mathcal D$ are convex bodies with the same measure $\mu_{\Om_1}=\mu=\mu_{\Om_2}$. Here the notation for a domain refers to its closure. For any $\nu\in [0,1]$, let
$$ \Omega_{\nu} = (1-\nu)\Omega_{1} + \nu\Omega_{2} \quad \text{and}\quad
m(\nu) = \T(\Omega_{\nu})^{\frac{1}{(n-p+1)}}.$$
Minkowski convexity in Definition~\ref{def:normalizedclass} ensures that this segment belongs to $\mathcal D$. Using the assumed first-variation inequality \eqref{eq:Thadamard} and the chain rule at $\nu=0+$, we obtain
\begin{align*}
m'_+(0)=\lim_{t\to 0^+} \frac{m(t)-m(0)}{t}
&\leq \T(\Omega_1)^{\frac{p-n}{n-p+1}}
\int_{\Snn}(h_{\Omega_2}-h_{\Omega_1})\,d\mu_{\Omega_1}\\
&=\T(\Omega_1)^{\frac{p-n}{n-p+1}}
\int_{\Snn}(h_{\Omega_2}-h_{\Omega_1})\,d\mu\\
&=\T(\Omega_1)^{\frac{p-n}{n-p+1}}
\left(\int_{\Snn}h_{\Omega_2}\,d\mu_{\Omega_2}
-\int_{\Snn}h_{\Omega_1}\,d\mu_{\Omega_1}\right)\\
&=\frac{m(1)^{n-p+1}-m(0)^{n-p+1}}{m(0)^{n-p}}.
\end{align*}
Now, the Brunn-Minkowski inequality \eqref{eq:bm0} implies that
$\nu\mapsto m(\nu)$ is a concave function on $[0,1]$ and, therefore,
\begin{equation}\label{eq:mconcave}
	m'_+(0)\geq m(1)-m(0),
\end{equation}
with equality if and only if $\nu\mapsto m(\nu)$ is linear on $[0,1]$.
Let
\[
q=n-p+1>0\qquad\text{and}\qquad a=\frac{m(1)}{m(0)}.
\]
The preceding upper bound for $m'_+(0)$ and \eqref{eq:mconcave} give
\[
m(1)-m(0)\leq m'_+(0)\leq \frac{m(1)^q-m(0)^q}{m(0)^{q-1}},
\]
or, equivalently,
\begin{equation}\label{eq:aqa}
	a^q\geq a.
\end{equation}
If $p<n$, then $q>1$, and \eqref{eq:aqa} implies $a\geq1$.
Arguing by changing the order of $\Om_1$ and $\Om_2$, using the other first-variation inequality in \eqref{eq:Thadamard}, we obtain $a\leq1$.
If $n<p<n+1$, then $0<q<1$, and \eqref{eq:aqa} implies $a\leq1$.
Changing the order gives $a\geq1$. Thus, whenever $p\neq n$, we have
$m(0)=m(1)$. The squeeze above then yields
$0=m(1)-m(0)\leq m'_+(0)\leq 0$, so equality holds in \eqref{eq:mconcave}, and hence
$m$ is constant on $[0,1]$.
It remains to consider $p=n$, for which $q=1$. In this case the squeeze gives directly
\[
m(1)-m(0)\leq m'_+(0)\leq m(1)-m(0),
\]
so $m'_+(0)=m(1)-m(0)$.
Thus, equality again holds in \eqref{eq:mconcave}, and the concavity of $m$
implies that $m$ is linear on $[0,1]$ (although it need not be constant).
Consequently, for every $1<p<n+1$,
\[
m(\nu)=(1-\nu)m(0)+\nu m(1),\qquad 0\leq\nu\leq1.
\]
This means that equality holds in the Brunn-Minkowski inequality
\eqref{eq:bm0} for $\Om_1$ and $\Om_2$. From the equality case of Theorem
\ref{thm:bm}, they are dilates about the chosen origin, say $\Om_2=b\Om_1$ for some $b>0$. Since
\[
\mu=\mu_{\Om_2}=b^{n-p}\mu_{\Om_1}=b^{n-p}\mu,
\]
and $\mu$ is nonzero, $b=1$ if $p\ne n$. The proof is finished.
\end{proof}

\subsection{Concluding remarks}
Here we list some final remarks on the results obtained:

\begin{enumerate}
\item The restriction $p<n+1$ is natural and is similar to $p<n$ for the Capacitary case. 
For the case $p=n+1$, we expect that a quantity similar to the $n$-dimensional logarithmic capacity as in Colesanti-Cuoghi \cite{CC}, could be found that satisfies a Brunn-Minkowski inequality. For the case $n+1<p<\infty$, we expect a Brunn-Minkowski inequality to hold for a functional analogous to the case for $n<p<\infty$ as in
\cite{ALSV} using techniques of $\mathcal A$-harmonic Green's functions. We also expect our results to be generalizable to $\mathcal A$-harmonic measures although there could be more involved technical aspects.

\item Further characterization of domains satisfying such conditions may depend on
specific choices of $p$-harmonic measures or
more involved quantitative characterization of the convexity and the nature of the growth of level sets of $p$-harmonic functions and estimates of Gaussian curvature of their boundaries, see for example \cite{Ma-Ou-Zh}.

\end{enumerate}
\section*{Acknowledgement}
During the preparation of this work, the authors used ChatGPT to check mathematical arguments and improving the language and presentation. The authors independently reviewed and verified the mathematical content and edited the text. They take full responsibility for the content of the published article.

\bibliographystyle{plain}
\bibliography{minkprob}

\end{document}